\documentclass{birkmult}
 \theoremstyle{definition}
 
 \theoremstyle{remark}

 \numberwithin{equation}{section}

 \usepackage{amsfonts,amsmath,latexsym,amsthm,amssymb,graphicx,xypic}
\usepackage[usenames,dvipsnames]{color}

\usepackage{amstext}
\usepackage{cite}
\usepackage{color}
\usepackage{pdfsync}
\usepackage{hyperref}

\newcommand{\RR}{\mathbb R}

\newcommand{\NN}{\mathbb N}
\newcommand{\CC}{\mathbb C}
\newcommand{\QQ}{\mathbb{Q}}
\newcommand{\calA}{\mathcal A}
\newcommand{\calK}{\mathcal K}

\newcommand{\calC}{\mathcal C}
\newcommand{\calH}{\mathcal H}

\newcommand{\calJ}{\mathcal J}
\newcommand{\calI}{\mathcal I}
\newcommand{\del}{\partial}
\newcommand{\dM}{\partial M}
\newcommand{\U}{\overline{\cup}}
\newcommand{\bff}{{b\!f}}
\newcommand{\imspec}{\operatorname{Im\, spec}}
\newcommand{\Specb}{\operatorname{Spec}_b}

\newcommand{\frakd}{\mathfrak{d}}

\newcommand{\ff}{{f\!\!f}}
\newcommand{\lf}{{l\!f}}
\newcommand{\rf}{{r\!f}}

\newcommand{\Id}{\mathrm{Id}}
\newcommand{\phg}{{\mathrm {phg}}}

\newcommand{\calV}{\mathcal{V}}
\newcommand{\calE}{\mathcal{E}}

\newcommand{\calG}{\mathcal{G}}
\newcommand{\calD}{\mathcal{D}}

\newcommand{\eps}{\epsilon}
\newcommand{\interior}[1]{\mathring{#1}}
\newcommand{\Span}{\operatorname{span}}
\newcommand{\id}{\operatorname{id}}

\newcommand{\Diff}{\operatorname{Diff}}

\newcommand{\Nhat}{\hat{N}}

\newcommand{\End}{\operatorname{End}}
\newcommand{\Ptilde}{\tilde{P}}

\renewcommand{\Re}{\operatorname{Re}}

\renewcommand{\split}{{\mathrm{split}}}
\newcommand{\Piperp}{\Pi^\perp}

\newcommand{\Mint}{\interior{M}}
\newcommand{\ext}{\text{ext}}
\newcommand{\bfy}{{\bf y}}
\newcommand{\bfz}{{\bf z}}

\newcommand{\Rtilde}{\tilde R}
\newcommand{\Diag}{\operatorname{Diag}}

\newcommand{\cphiT}{{}^{c\phi}T}

\newcommand{\calKbar}{\overline{\calK}}

\newcommand{\bT}{{}^{b}T}
\newcommand{\Cinfty}{C^\infty}
\newcommand{\bphi}{{b\phi}}
\newcommand{\bfR}{{\bf R}}
\newcommand{\bfA}{{\bf A}}
\newcommand{\II}{{\mathbb{II}}}

\newcommand{\dvol}{\operatorname{dvol}}
\newcommand{\sus}{{\mathrm{sus}}}

\newcommand\datver[1]{\def\datverp
{\par\boxed{\boxed{\text{Version: #1; Run: \today}}}}}
\datver{1.1; Revised: 4 Feb 08 by E.}

\newtheorem{theorem}{Theorem}

\newtheorem{proposition}{Proposition}

\newtheorem{definition}{Definition}

\begin{document}
%
%
%
%
%
%
%
%
%
\title[A Parametrix Construction for the Laplacian on $\QQ$-rank 1 Locally Symmetric Spaces]
 {A Parametrix Construction for the Laplacian on\\ $\QQ$-rank 1 Locally Symmetric Spaces
 }
\author{D. Grieser}
\address{Institut f\"ur Mathematik\br
Carl von Ossietzky Universit\"at Oldenburg\br
26111 Oldenburg\br
Germany}
\email{daniel.grieser@uni-oldenburg.de}
\author{E. Hunsicker}

\address{
Department of Mathematical Sciences\\
Loughborough University\\
Loughborough\\
LE11 3TU \\
UK}

\email{E.Hunsicker@lboro.ac.uk}

\thanks{This work was completed with the support of Leverhulme
  Trust Project Assistance Grant F/00 261/Z}



\begin{abstract}
This paper presents the construction of parametrices for the Gauss-Bonnet and Hodge Laplace operators on noncompact manifolds modelled on $\QQ$-rank 1 locally symmetric spaces.
These operators are, up to a scalar factor, $\phi$-differential operators, that is, they live in the generalised $\phi$-calculus studied by the authors in a previous paper, which extends work of Melrose and Mazzeo.
However, because they are not
 totally elliptic elements in this calculus, it is not possible to construct parametrices for these operators within
 the $\phi$-calculus.  
We construct parametrices for them in this paper using a combination of the $b$-pseudodifferential 
operator calculus of R. Melrose and the $\phi$-pseudodifferential operator
 calculus.  The construction simplifies and generalizes the construction done by Vaillant in his thesis for the Dirac operator.
 In addition, we study the mapping properties of these operators
 and determine the appropriate Hlibert spaces between which the Gauss-Bonnet and Hodge
 Laplace operators are Fredholm.  Finally, we establish regularity results for elements of the kernels of
these operators.
 \end{abstract}

\maketitle
\tableofcontents

\section{Introduction}
Analysis of the Laplacian on locally symmetric spaces has interesting applications in 
analytic number theory and in the relationships between analysis and topology on 
singular spaces.  An important tool for studying the Laplacian is the method of pseudodifferential
operators.  The basic philosophy of this approach is to define a space of 
pseudodifferential operators and one or more symbol maps on this space whose invertibility
defines a subset of (totally) elliptic elements for which parametrices can be constructed within
the calculus.  Ideally, the operator you are interested in studying, for us the Laplacian, 
will be a (totally) elliptic element in your calculus, which will then provide the necessary
tools for studying it.

In this paper, we are interested in the Gauss-Bonnet operator and its square, the 
Hodge Laplace operator, on $L^2$ differential forms over Riemannian
manifolds modelled on $\mathbb{Q}$-rank 1 locally symmetric spaces.  
These spaces can be naturally compactified to manifolds with boundary.
In this paper, we will refer to such manifolds with boundary as $\phi$-manifolds, and to their Riemannian metrics as $\phi$-cusp
metrics.  A $\phi$-manifold, $M$, 
will be a compact orientable manifold 
with boundary, $\partial M=Y$, which is the total space of a fibre
bundle $Y \stackrel{\phi}{\to} B$.  A $\phi$-cusp metric will be a metric $g$ on the interior $\interior M$ that satisfies the property that near the boundary of $M$, with respect to some trivialization of a neighborhood $U$ of the boundary $U \equiv [0,\epsilon)\times \partial M$,
it can be written in the form
\begin{equation}\label{eq1}
g = \frac{dx^2}{x^2} + \phi^*g_B + x^{2a}h.
\end{equation}
Here $a \in \mathbb{N}$, $x$ is the coordinate on $[0,\epsilon)$ and $h$ is a symmetric 2-tensor
on $\del{{M}}$ that restricts to a metric on each fibre of $\phi$, and $g_B$ is a Riemannian metric on $B$.  
Also, we assume that $\phi:Y\to B$ is a Riemannian fibration with respect to the metrics $\phi^*g_B + x^{2a}h$ on $Y$ and $g_B$ on $B$, for some (hence all) $x\in(0,\eps)$. 

In a previous paper \cite{GH1}, we constructed a calculus of pseudodifferential operators
adapted to the geometric setting of generalized $\phi$-manifolds with general $\phi$-cusp metrics
(which may have a lower order perturbation from the metrics considered in this paper).
This calculus, which we called the small ${\phi}$-calculus, is a generalisation of the 
$\phi$-calculus of Melrose and Mazzeo, \cite{MaMe}, in that it permits a broader possible set 
of metric degenerations than that calculus does.  In particular, the Melrose and Mazzeo calculus considers
operators associated to metrics as in equation \eqref{eq1} where $a=1$, whereas we allow $a$
to be any natural number. Also, we allow a stack of several fibrations at the boundary with different orders of degeneration.
A disappointing aspect of the small $\phi$-calculus in both the Melrose-Mazzeo setting and our more general
setting is that, although standard geometric operators such as Dirac
operators, the Laplace-Beltrami operator and the Hodge Laplacian are elements of this calculus, they
are typically not ``totally elliptic" elements, which means that it is not possible to find good parametrices for them
within this calculus.
However, the Hodge Laplacian has an additional structure reflecting the boundary fibration, which we can exploit to construct a parametrix in a larger $\phi$-calculus.

The parametrix construction combines aspects of the original $b$-calculus of Melrose and 
the $\phi$-calculus of \cite{GH1}.
The idea of combining elements of these two calculi to construct a parametrix was used 
first by Boris Vaillant in his thesis 
\cite{Va} for the Dirac operator on $\phi$-manifolds with $\phi$-cusp metrics where $a=1$ under certain geometric conditions.  The goal of this paper and the upcoming
paper \cite{GH3} is to simplify and generalise this construction, fully understand the space of operators
in which a  parametrix can be built, and characterise the conditions on $\phi$-differential operators that allow this approach to work.
One of the main differences between Vaillant's work and the work in this paper is that certain geometric obstructions arise in the construction of the parametrix of the (second order) Hodge Laplacian that do not arise in the construction of the parametrix
of the (first order) Dirac operator.  These obstructions have arisen in work of J. Mueller
\cite{JM} that studies the Hodge Laplacian on $\phi$-manifolds from 
a perturbation theory viewpoint.  Thus it is interesting to see how the obstructions also
arise using a pseudodifferential approach.

Many $\mathbb{Q}$-rank 1 locally symmetric spaces compactify to $\phi$-manifolds with $\phi$-cusp metrics.  However,
a general $\mathbb{Q}$-rank 1 locally symmetric space may compactify to a manifold
whose boundary is the total space of a double fibration as opposed to a single fibration.  
In a work in preparation,  \cite{GH3}, we consider this more general case, and in 
addition generalise our results to other 
elliptic geometric operators on such manifolds. However, the main ideas of this general work can already
be largely understood through the special case we consider in this paper, 
and the work of extending is quite technical.  Thus the case we consider here is useful not only
in its own right, but also for grasping the essential points of the general construction done in \cite{GH3}.
This work is a step towards the development
of an extended pseudodifferential operator calculus on locally symmetric spaces of any rank
and their generalisations.  

In order to state the main theorems, we need to introduce some notation.
A fundamental object in the analysis of the Gauss-Bonnet and Hodge Laplace operators is the bundle $\calKbar\to B$ of fibre-harmonic forms. This is the finite dimensional vector bundle over $B$ whose fibre $\calKbar_\bfy$ over $\bfy\in B$ is the space of harmonic forms on the fibre $F_\bfy=\phi^{-1}(\bfy)$, with respect to the metric $h_{|F_\bfy}$.
 We call a differential form on $\interior M$ fibre harmonic if its restrictions to all fibres over $U$ are harmonic, and fibre perpendicular if the restrictions are perpendicular, in the $L^2$ space of the fibre, to the harmonic forms. Fibre harmonic forms may be regarded as forms on $[0,\eps)\times B$ valued in $\calKbar$.  Our operators behave differently on the subspaces of fibre harmonic forms and  fibre perpendicular forms, and this difference is the main difficulty in the construction. It is reflected in the structure of the Sobolev spaces on which the operators are Fredholm. We call them split Sobolev spaces. They are denoted 
 $H^m_\split(M,\Lambda T^*M,\dvol_b)$ and defined in Section \ref{subsec:split sob}. Here the form bundle $\Lambda T^*M$ is endowed with the metric induced by $g$, and $\dvol_b$ is a b-volume form, that is, a volume form on $\interior M$ which near the boundary is $\frac1x$ times a smooth volume form on $M$.
 
The metric $g$ induces a flat connection on the bundle $\calKbar$, and this can be used to define a Gauss-Bonnet operator $D_V$ on $V=[0,\eps)\times B$ acting on forms with coefficients in $\calKbar$, see \eqref{eqn:D_V}. This is an elliptic b- (or totally characteristic) operator, hence has a discrete set of critical weights associated with it, for which the operator is not Fredholm between the natural weighted Sobolev spaces. We call this set $-\imspec (D_V)$, see \eqref{eqn:def imspec}.  This set also arises in the expansions  of harmonic forms over $M$ near the boundary.
Recall that a smooth form $u$ on $\interior M$ is called polyhomogeneous if at the boundary it has a full asymptotic
expansion  of the form:
\begin{equation}
\label{eqn:def phg} 
 u\sim \sum_{w,k} x^w (\log x)^k u_{w,k}\, ,\quad x\to 0 
\end{equation}
where all $u_{w,k}$ are smooth up to $x=0$ and the sum runs over a discrete set of $w\in\CC$ with $\Re w\to\infty$ and $k\leq N_w$ for each $w$, for some $N_w\in\NN_0$.  The pairs $(w,k)$ that arise in 
the expansions of harmonic forms can be derived from the set $-\imspec (D_V)$. Here smoothness of $u_{w,k}$ at $x=0$ is to be understood as smoothness of the coefficient functions when writing the form in terms of $\frac{dx}x, dy_j, x^a dz_k$ where the $y_j$ are base coordinates and the $z_k$ are fibre coordinates. This may be stated as smoothness when considering $u_{w,k}$ as section of a rescaled bundle $\Lambda \cphiT^*M$, as explained in Section \ref{sec:laplacian}.

Our first main result is a general parametrix construction, which is stated as Theorem \ref{thm:right split param}. Using this parametrix we deduce the following results.
\begin{theorem}
\label{Thm1}
Let $M$ be a $\phi$-manifold endowed with a $\phi$-cusp metric, $g$. 
The Gauss-Bonnet operator, $D_M=d+d^*$, is a Fredholm operator
\[
D_M: x^{\gamma+a} H^1_\split(M,\Lambda T^*M,\dvol_b) \rightarrow x^{\gamma}L^2(M,\Lambda T^*M,\dvol_b)
\]
for every $\gamma \notin -\imspec (D_V)$.

If  $D_Mu = 0$ for $u \in L^2(M,\Lambda T^*M,\dvol_b)$, then $u$ is polyhomogeneous, and in the expansion \eqref{eqn:def phg} we have $\Re w > 0$ for all $w$ and all terms $u_{w,k}$ with $\Re w \leq a$ are fibre harmonic.

\end{theorem}

In the next theorem we refer to operators $D_B$ and $\Pi$. The operator $D_B$ is a kind of Gauss-Bonnet operator on $B$ with coefficients in all forms on $F$, see \eqref{eqn:DU DB DF},  and $\Pi$ is the fibrewise orthogonal projection to the fibre-harmonic forms. 
\begin{theorem}
\label{Thm2}
Let $M$ be a $\phi$-manifold endowed with a $\phi$-cusp metric, $g$, and assume that $[D_B, \Pi]=0$.
Then the Hodge Laplacian, $\Delta_M=D_M^2$, is a Fredholm operator 
\[
\Delta_M: x^{\gamma+2a} H^2_\split(M,\Lambda T^*M,\dvol_b) \rightarrow x^{\gamma}L^2(M,\Lambda T^*M,\dvol_b).
\]
for every $\gamma \notin -\imspec (D_V)$.

If $\Delta_M u = 0$ for $u \in L^2(M,\Lambda T^*M,\dvol_b)$,  then $u$ is polyhomogeneous, and in the expansion \eqref{eqn:def phg} we have $\Re w > 0$ for all $w$ and all terms $u_{w,k}$ with $\Re w \leq  2a$ are fibre harmonic.

\end{theorem}

This paper is organised as follows.  In Section \ref{sec:b+phi-calculi}, we review the general b-calculus approach to pseudodifferential operators and give definitions, parametrix and regularity theorems in the 
b and small-$\phi$-calculus settings.  These results are useful both for comparison to Theorems
\ref{Thm1} and \ref{Thm2}, and also because parts of them are used in the proofs of these theorems.
In Section \ref{sec:laplacian}, we review the geometry of fibrations and the form of the Gauss-Bonnet and Hodge
Laplace operators over $\phi$-manifolds with $\phi$-cusp metrics.  This section lays out some of the
critical properties of these operators which permit the approach taken in the proofs of Theorems
\ref{Thm1} and \ref{Thm2}. Here also the rescaled bundle $\cphiT^*M$ is introduced, which is needed to apply the $b$-calculus techniques and is used in the rest of the paper. It also explains where the condition $[D_B,\Pi]=0$ in Theorem \ref{Thm2} comes from.
In Section \ref{sec:ext phi calculus}, we consider how to lift integral kernels of elements of the $b$-calculus to the $\phi$-double space, and how the resulting operators combine with elements of the $\phi$-calculus.  
A larger space of operators, the ``extended" $\phi$-calculus, in which both $\phi$-operators 
and lifted $b$-operators live and may be combined, is defined.  It is also shown that the 
extended $\phi$-calculus has a meaningful boundary symbol, and that it has properties similar
to the original $\phi$-calculus that will be used in the construction of parametrices.  Section \ref{sec:parametrix} begins with a discussion of the properties we want from our final parametrix, then contains the statement and proof of the main (parametrix) theorem of this paper, Theorem \ref{thm:right split param}.  Section \ref{sec:proofs main thms} contains the definition of the spaces that arise
in the Fredholm results in Theorems \ref{Thm1} and \ref{Thm2}, and proofs of general Fredholm and regularity
theorems as corollaries of the parametrix theorem.  These theorems, when applied to the 
Gauss-Bonnet and Hodge Laplace operators, become Theorems \ref{Thm1} and \ref{Thm2}.
This paper is written in such a way that a reader interested primarily in applying the theorems, rather
than in the proof details, can skip the more technical sections 4 and 5.


\section{The b- and $\phi$-calculi}
\label{sec:b+phi-calculi}
In the b-calculus approach to geometric pseudodifferential 
operators, the geometry of a singular or noncompact manifold is encoded in a Lie algebra of vector
fields over a manifold with boundary that degenerate in a particular way at the boundary.
The original b-setting of Melrose, \cite{Me-aps}, dealt with a Riemannian manifold $(\interior{M}, g_b)$ 
which off of a compact set had an infinite cylindrical metric,
\[
g_b = dr^2 + g_Y, \quad r \in [0,\infty),
\]
where $(Y,g_Y)$ was a smooth compact Riemannian manifold.  Certain perturbations of such product type metric are also considered. Under the change of coordinates $x=e^{-r}$,
the metric can be rewritten in the form:
\[
g_b = \frac{dx^2}{x^2} + g_Y, \quad x \in (0, 1].
\]
We compactify $\interior{M}$ by adding a copy of $Y$ at $x=0$, thus obtaining a compact manifold with boundary $M$ whose interior is $\interior{M}$. The singular behaviour of the metric at the boundary $x=0$ induces 
degenerations in the associated geometric differential operators at $x=0$. These operators are defined
on $M$ including its boundary, and are expressible, in a smooth and non-degenerate way, in terms of vector fields tangent to the boundary. In terms of local coordinates $\{x,{\bf y}\}=\{x,y_1, \ldots, y_n\}$ near a  boundary point these vector fields are spanned by
\begin{equation}\label{eq:bvf}
x\del_x, \del_{y_1}, \ldots, \del_{y_n}\,.
\end{equation}
More precisely, the geometric operators are b-differential operators. The set of b-differential operators of order $m$,
denoted $\mbox{Diff}^m_b(M)$, consists of order $m$ differential operators on 
$M$ which have smooth coefficients in the interior and in coordinates near $x=0$ can be written in the form
\begin{equation}
\label{eq:bdiffop}
P=\sum_{j+|K| \leq m} a_{j, K}(x,{\bf y}) (x\calD_x)^j (\calD_{\bf y})^K,
\end{equation}
where the functions $a_{j,K}$ are smooth up to the boundary $x=0$, 
$K$ is a multi-index, $\calD_x=\frac1i \frac\partial{\partial x}$ and
$(\calD_{\bf y})^{K} = \calD_{y_1}^{K_1} \cdots \calD_{y_n}^{K_n}$.
Similarly, for a vector bundle $E$ over $M$ we say $P\in\Diff_b^m(M,E)$ if $P$ has, in local coordinates and with respect to a local bundle trivialization $E\cong \RR^e$, the form \eqref{eq:bdiffop} where the $a_{j,K}(x,\bfy)$ are homomorphisms $\RR^e\to \RR^e$. Then $P$ acts on sections of $E$.
\medskip

In the setting of this paper, the degeneration of the $\phi$-cusp metric is encoded by smooth vector fields on $M$ that near any boundary point
may be written as a smooth linear combination of vector fields of the form
\begin{equation}\label{degvf}
x^{1+a}\del_x, x^{a}\del_{y_1}, \ldots, x^{a}\del_{y_b}, \del_{z_1}, \ldots, \del_{z_f},
\end{equation}
where the boundary is $x=0$, the coordinates $y_i$ are lifted from the base $B$ of the fibration $Y\stackrel{\phi}{\to}B$  and the $z_j$ can be thought of 
as local fiber coordinates.  
Thus, a $\phi$-differential operator of order $m$ is an $m$-th order differential operator on 
$M$ which near any boundary point can be written in the form:
\begin{equation}
\label{eq:phidiffop}
P=\sum_{j+|K|+|L| \leq m} a_{j, K,L}(x,{\bf y}, {\bf z}) (x^{1+a}\calD_x)^j
(x^{a}\calD_ {\bf y})^K (\calD_ {\bf z})^L,
\end{equation}
where the functions $a_{j, K,L}(x,{\bf y}, {\bf z})$ are smooth up to the boundary $x=0$.
We denote the algebra of $\phi$-differential operators over $M$ by $\mbox{Diff}^*_\phi(M)$ or $\Diff^*_\phi(M,E)$.
Then if $A$ is a geometric differential operator of order $m$ over $M$ associated to
the metric $g$, we have
\[
A = x^{-ma}P, \qquad P \in \mbox{Diff}^m_\phi(M,E)
\]
where $E$ is typically a variant of the form bundle $\Lambda T^*M$, see Section \ref{sec:laplacian}. 
Here and in the sequel $x\in\Cinfty(M)$ always denotes a boundary defining function for $M$; locally near any boundary point it is a coordinate, and in the interior of $M$ it is positive. 
It is quite reasonable to ask why we have put the factor of $x^{-ma}$ out front, rather than
building up our differential operators from vector fields of the forms $x\del_{x},$ $\del_{y_i}$
and $x^{-a} \del_{z_j}$.  The reason is that those vector fields do not span a Lie
algebra over smooth functions on $M$, whereas the vector fields in equation \eqref{degvf} do.

Differential operators, and the broader class of pseudodifferential operators associated
to a Lie algebra of vector fields, are identified with their Schwartz kernels. These are distributions on the double space, $M^2$, that are singular at the diagonal. In order to separate this singularity from the degeneration at the boundary it is useful to lift these kernels  to a particular model space obtained by a blow-up of $M^2$, where their structure can be described explicitly. The space $M^2$ and these model spaces are manifolds with corners.  The space $M^2$ has two boundary hypersurfaces, $\lf=\partial M\times M$ and $\rf=M\times \partial M$. In the b-calculus setting, the model space $M^2_b$ has 
three boundary hypersurfaces, $\lf,\rf$ and  $\bff$.  In the $\phi$ setting, the model space $M^2_\phi$ 
has four boundary hypersurfaces: $\lf,\rf,\bff$ and $\ff$. The diagonal of $M^2$ lifts to $M^2_b$ and $M^2_\phi$.
Operators 
in the full calculi in both settings are defined by distributions on the appropriate model space 
that are conormal to the lifted diagonal and have polyhomogeneous expansions at each boundary face.  
The degree of the conormality of a kernel at the lifted diagonal determines the order of its associated
operator, as in the setting of standard pseudodifferential operators.  
Its expansions at the various boundary faces determine how it maps between
appropriate weighted Sobolev spaces.  

In order to state the mapping results for elements of the full
b and $\phi$ calculi that we will need in our parametrix construction, we need first to define weighted b and $\phi$ Sobolev spaces, as well as  spaces of polyhomogeneous and conormal distributions.  Then we will move on to a summary of the necessary results from the b and $\phi$-calculi.
In the interests of space, we will give
only heuristic explanations, ones we find the most useful in applications, and refer the reader to the references
\cite{Me-aps}, \cite{GH1} and \cite{Gbbc} for details of exactly how and why these definitions and theorems work. 
Note that the notation we adopt here conforms to that in \cite{Me-aps} but not quite to that in \cite{GH1}. For example, the space denoted $\Psi_b^{m,\calI}(M,\Omega_b^{1/2})$ below is denoted 
$\Psi_b^{m,\calI}(M)$ in \cite{GH1}.

\subsection{Sobolev spaces and polyhomogeneous spaces}
Throughout the paper we fix a b-volume
form $\dvol_b$ on $\interior M$. The index b indicates that we assume that in local coordinates near a boundary point, we can write
$$ \dvol_b = a(x,\bfy) \frac{dx}x dy_1,\dots,dy_n $$
where $a$ is smooth and positive up to the boundary $x=0$.
While this b-behavior near the boundary is important, the particular choice of $\dvol_b$ is inessential. 
We define Sobolev spaces always with respect to this volume form.

Elliptic b-differential operators map as Fredholm
operators between appropriate b-Sobolev spaces, defined as follows.  
Let $E$ be a vector bundle over $M$. 
Choose a cutoff function $\chi$ which is equal to $1$ near $x=0$ and equal to $0$ for $x>\eps$.  Then we define for $m\in\NN_0$, $s\in\RR$
\begin{equation}
\label{eq:bSobolev}
x^sH^m_b(M,E,\dvol_b) = \{ x^s u\mid u \in H^m_{loc}(\interior M,E), (x\calD_x)^j (\calD_\bfy)^K \chi u \in L^2(M, E,\dvol_b)\,\, \forall j+|K| \leq m \}.
\end{equation}
These spaces are independent of the choice of $\dvol_b$ and cutoff $\chi$, local coordinates and local trivializations of $E$, and metric on $E$, and can be metrized using invertible
b-pseudodifferential operators, but we will not go into these details here.  
This definition can be extended to $m\in\RR$ by the usual arguments of duality and interpolation. For most applications $m\in\NN_0$ is sufficient, but we state theorems in this greater generality.
Similarly, the appropriate Sobolev spaces for $\phi$-operators are expressed in terms of the vector fields in equation \eqref{degvf}, so they are defined as follows:
\begin{equation}
\label{eq:phisob}
x^sH^m_{\phi}(M,E,\dvol_b)= \{x^su \mid u \in H^m_{loc}(\interior M,E), \qquad \qquad \qquad \qquad
\end{equation}
\[
\qquad \qquad \qquad (x^{1+a}\calD_x)^j (x^a \calD_\bfy)^K (\calD_\bfz)^L \chi u \in L^2(M,E, \dvol_b)\,\, \forall j+|K|+|L|  \leq m \}.
\]
Again, these can be metrized using invertible $\phi$-differential operators, as is carried out in 
\cite{GH1}.

We next need to define polyhomogeneous sections, which are
sections that have asymptotic expansions in terms of the boundary defining function $x$ near $\del M$.
The expansions near the boundary for these sections involve powers of $x$ and powers of $\log x$.
Thus the type of these expansions can be described by index sets $I$, which are sets of pairs
$(z,k) \in \mathbb{C} \times \mathbb{N}_0$, where the first term describes permitted powers of $x$
and the second describes permitted powers of $\log x$. Index sets are required to satisfy $(z,k)\in I, l\leq k\Rightarrow (z,l)\in I$ and $(z+1,k)\in I$ (this guarantees that the condition below is independent of the choice of boundary defining function $x$), and that for any $r$ there is only be a finite number of $(z,k)\in I$ satisfying $\Re z < r$.
Polyhomogeneous
sections over $M$ are defined as follows.
\begin{definition}
\label{def:phg}
Let $I$ be an index set and $E$ a vector bundle over $M$, a manifold with boundary.
A smooth section of $E$ over $\interior{M}$ is said to belong to $\mathcal{A}^I_{\rm{phg}}(M,E)
$ if  for
each $(z,k) \in I$ there exists a section $u_{(z,k)}$ of $E$, smooth up to the boundary of $M$,  such that for all $N$, the difference
$$
u-\sum_{\substack{(z,k) \in I \\ \Re(z) \leq N}} x^z (\log x)^k u_{(z,k)}
$$
vanishes to $N$th order at $\del M$.

For an index set $I$, we say $I>\alpha$ if  $(z,k)\in I$ implies $\Re z>\alpha$, and $I \geq \alpha$ if  $(z,k)\in I$ implies $\Re z \geq \alpha$, and $k=0$ for $\Re z = \alpha$.

\end{definition}
When operators are composed
or applied to sections, their index sets combine using two operations.  The first
is simple set addition, indicated by $+$.  The second is an extended union:
$$
I \U J = I \cup J \cup \{(z,k) \mid k=l_1 + l_2+1 \mbox{ where there exist } (z,l_1) \in I, (z, l_2) \in J \}.
$$
For the index set $(r+\NN_0)\times \{0\}$, where $r\in\RR$, we use the short notation $r$. Thus, $I+r$ is the index set $I$ shifted by $r$.  Note that for the index set $\emptyset$, we have
$I \U \emptyset = I$, $I + \emptyset = \emptyset$ and $\emptyset + r = \emptyset$.


There is an analogous definition of polyhomogeneity for sections on manifolds with corners, for example
the model double spaces $M^2_b$, $M^2_\phi$. In this case an {\em index family} $\calI$ has to be specified, which consists of an index set $\calI_f$ for each boundary hypersurface $f$ of the space.
Also, the definition of polyhomogeneity extends to distributions on $\interior M$ which are conormal with respect to a submanifold of $M$ that meets the boundary in a product type fashion. The main example of such a submanifold is the lifted diagonal as a submanifold of $M^2_b$ or $M^2_\phi$.
  Rather than go into the detail for this here, we refer the reader to 
\cite{Me-aps}, \cite{GH1}  and \cite{Gbbc}.

\subsection{Mapping and regularity in the b-calculus}

As mentioned above, Schwartz kernels of b-differential and b-pseudodifferential operators are seen
as living not on the double space, $M^2$, but rather on a blow-up of this space, called $M^2_b$. When identifying an operator with its Schwartz kernel there is integration, hence a choice of measure or density involved, unless one considers the operator as acting on densities, or takes the kernel to be a density. Another possibility, which is more symmetric, is to think of both as half-densities. It is important to deal with this carefully because of the various  singular coordinate changes involved in the blow-up of $M^2$ to $M^2_b$ and $M^2_\phi$. 
We will  go into this only shortly, since it is needed to define b- and $\phi$-pseudodifferential operators properly, but in the end the half-densities will only act behind the scenes and not be visible.

Recall that a density on an oriented manifold $M$ is a smooth top degree form, that is, a smooth section of the  trivial 1-dimensional vector bundle  $\Omega=\Lambda^{\dim M}T^*M$.  For us more important is the b-density bundle over an oriented manifold with boundary, $M$, denoted $\Omega_b$. The volume form $\dvol_b$ is a smooth non-vanishing section of $\Omega_b$. We can form the square root of this bundle, denoted $\Omega_b^{1/2}$.
Near a boundary point its sections are of the form   $u(x,\bfy) \left(\frac{dx}x\,dy_1\dots dy_n\right)^{1/2}$ in terms of (oriented) local coordinates $\{x,\bfy\}$, where the function $u$ is smooth up to the boundary. Similarly, one has the b-half-density bundle on a manifold with corners.

We now define the space of b-pseudodifferential operators on a manifold with boundary, $M$. Recall that $M^2_b$ is obtained by blowing up the corner $\partial M\times \partial M$ of $M^2$. The front face of the blow-up is called $\bff$, so $M^2_b$ has three boundary hypersurfaces, $\lf,\rf$ and $\bff$. Let $\mathcal{I}=(\calI_\lf,\calI_\rf,\calI_\bff)$  be an index family for $M^2_b$. 
For $m\in\RR$ we define the full b-calculus
$\Psi_b^{m,\mathcal{I}}(M,\Omega_b^{1/2})$ as the set of $\Omega_b^{1/2}$-valued distributions on $M^2_b$ which are conormal with respect to the lifted diagonal of order $m$ and polyhomogeneous at the boundary with index family $\calI$. These operators act on b-half-densities on $M$.  To extend this to operators acting on sections of a vector bundle $E$ over $M$ one proceeds in two steps. First, by tensoring kernels with sections of the $\End(E)$ bundle on $M^2$ lifted to $M^2_b$ we obtain operators acting on 
sections of $\Omega_b^{1/2}\otimes E$. Then we replace $E$ by $E\otimes \Omega_b^{-1/2}$ in this construction and obtain operators acting on sections of $E$. 

This space of kernels or operators is denoted by $\Psi_b^{m,\calI}(M,E)$. For $\mathcal{I}=(\emptyset, \emptyset, 0)$, i.e., kernels
vanishing to infinite order at $\lf$ and $\rf$ and smooth transversally to $\bff$, we simply write $\Psi_b^{m}(M,E)$. This is called the small b-calculus.
A simple calculation shows that $\Diff_b^m(M,E)\subset \Psi_b^{m}(M,E)$. The fact that here the index set at $\bff$ is simply $0$ results from and justifies the use of  b-half-densities instead of regular densities.

 The expansions at $\rf$ and $\bff$ determine the domain of an  operator
 in $\Psi_b^{m,\calI}(M,E)$,
and the expansions at $\lf$ and $\bff$ determine its range. This is made precise in the following theorems:

\begin{theorem}[Boundedness and compactness for b-operators]
\label{thm:bdd cpct b}
Let $M$ be a compact manifold with boundary and let $E$ be a vector bundle over $M$. Let $P\in \Psi_b^{m,\calI}(M,E)$ and $\alpha,\beta\in\RR$, $k\in\RR$.
\begin{enumerate}
 \item
  If 
\begin{equation}
 \label{eqn:condition bddness b}
 \calI_\lf>\beta,\ \calI_\rf >-\alpha,\ \calI_\bff \geq \beta-\alpha,
  \end{equation}
then $P$ is bounded as an operator
\begin{equation}\label{eqn:sobolev map b}
P: x^\alpha H^{k+m}_b(M,E,\dvol_b)  \to x^\beta H^{k}_b (M, E,\dvol_b). 
\end{equation}
\item
If  $m<0$ and strict inequality holds everywhere in \eqref{eqn:condition bddness b} then $P$, acting as in  \eqref{eqn:sobolev map b}, is compact.
\end{enumerate}
\end{theorem}
\noindent
Note that for the index set $\emptyset$, $\inf(\emptyset) = \infty$, so in particular, b-differential
operators are bounded between equal weighted b-Sobolev spaces of appropriate orders.
Next we can state the mapping property with respect to polyhomogeneous sections. Again the use of b-half densities makes the index sets combine in a very simple manner:
\begin{theorem}[Mapping of polyhomogeneous sections for b-operators]
\label{th:b-phgreg}
Let $u \in \calA^I_{\rm{phg}}(M,E)$ and $P\in \Psi_b^{m,\calJ}(M, E)$.  Then
if $\calJ_\rf + I> 0$, we can define $Pu$ and we get that
$Pu \in \calA^K_{\rm{phg}}(M,E)$, where
$K = \calJ_\lf \U (\calJ_\bff+I)$.
\end{theorem}

In order to state the parametrix and regularity results for $b$-elliptic operators, we 
need to discuss the two symbols that control ellipticity.
Associated to a b-differential operator \eqref{eq:bdiffop} is a b-symbol:
\begin{equation}
\label{eq:bprinc}
{}^b\sigma_m(P)(x,\bfy,\tau, \eta)=\sum_{j+|K| = m} a_{j, K}(x,{\bf y}) \tau^j \eta^K.
\end{equation}
Then $P$ is defined to be b-elliptic if ${}^b\sigma_m(P)(x,\bfy,\tau, \eta)$ is invertible for $(\tau, \eta) \neq 0$.
It turns out that this symbol is not enough to characterize Fredholm elements in the calculus,
so we need an additional ``boundary symbol" called the indicial operator.
For a b-differential operator as above, this is the differential operator defined by
$$
I(P)=\sum_{j+|K| \leq m} a_{j, K}(0,\bfy) (s\calD_s)^j \calD_\bfy^K.
$$
Here $(s, \bfy)\in \RR_+\times \partial M$, and $I(P)$ acts on sections of the bundle $E_{|\partial M}$ pulled back to $\RR_+\times \partial M$. 
Under the Mellin transform, this becomes a holomorphic family of operators on $\del M$, the indicial family:
$$
I(P, \lambda) = \sum_{j+|K| \leq m} a_{j, K}(0,\bfy) \lambda^j \calD_\bfy^K.
$$
The concepts of 
b-principal symbol, indicial operator and indicial family can be generalized to operators in $\Psi_b^{m,\mathcal{I}}(M, E)$.

\begin{definition}\label{def:Spec} Let $P\in\Psi_b^{m,\calI}(M,E)$ be b-elliptic. The sets $\mbox{Spec}_b(P) \subset \mathbb{C} \times \mathbb{N}_0$ and $-\imspec(P)\subset\RR$ are defined as:
$$
\Specb(P) = \{ (\lambda, k) \mid I(P, \lambda) \mbox{ is not invertible on }
C^\infty(\del M,E) \mbox{ and has a pole of order } k+1 \mbox{ at } \lambda \},
$$
\begin{equation}
 \label{eqn:def imspec}
-\imspec(P) = \{ -\mbox{Im}(\lambda) \mid I(P, \lambda) \mbox{ is not invertible on }
C^\infty(\del M,E) \}.
\end{equation}
\end{definition}
These sets are central in describing the mapping properties of the operator $P$.
The set $-\imspec(P)$ is a discrete subset of $\RR$.

We can now state the parametrix theorem for b-operators. In this paper we are not interested in the precise index sets, so we state it in the following form (cf. Proposition 5.59 in \cite{Me-aps}).
\begin{theorem}[Parametrix in the b-calculus] \label{thm: b-parametrix}
Let $P \in \mbox{Diff}^m_b(M,E)$ be b-elliptic. Then for each $\alpha \notin -\imspec(P)$ there is an index family $\calE(\alpha)$ for $M^2_b$ determined by $\Specb(P)$ and satisfying
$$ \calE(\alpha)_{\lf} > \alpha,\quad \calE(\alpha)_{\rf} > -\alpha,\quad \calE(\alpha)_{\bff} \geq 0 $$
and parametrices 
$$
Q_{b,r,\alpha},\ Q_{b,l,\alpha} \in \Psi_b^{-m, \calE}(M,E),
$$
such that
$$
P \circ Q_{b,r,\alpha} = \Id - R_{b,r,\alpha},\quad Q_{b,l,\alpha} \circ P= \Id - R_{b,l,\alpha},$$
where the remainders satisfy 
$$ R_{b,r,\alpha} \in x^\infty \Psi_b^{-\infty,\calE(\alpha)}(M,E),\quad
R_{b,l,\alpha} \in  \Psi_b^{-\infty,\calE(\alpha)}(M,E) x^\infty.
$$
\end{theorem}
Note that the $x^\infty$ factor on the left means that the kernel of $R_{b,r,\alpha}$ actually vanishes to infinite order at $\lf$ and $\bff$, while the $x^\infty$ factor on the right means that the kernel of $R_{b,l,\alpha}$ vanishes to infinite order at $\rf$ and $\bff$.

Combined with Theorems \ref{thm:bdd cpct b} and \ref{th:b-phgreg}, and a similar mapping result for the remainders, this gives (cf. Theorem 5.60 and Prop. 5.61 in \cite{Me-aps}):
\begin{theorem}[Fredholmness and regularity of elliptic b-operators]
\label{thm:b Fredholm regularity}
 Let $P\in \Diff^m_b(M,E)$ be  b-elliptic. 
 Then $P$ is Fredholm
as a map $P:x^\alpha H^{k+m}_b(M,E,\dvol_b) \rightarrow x^{\alpha}H^{k}_b(M,E,\dvol_b)$ for any $\alpha \notin -\imspec(P)$ and any $k\in\RR$.

Further, 
if $u \in x^\alpha H^k_b(M,E,\dvol_b)$ for some $\alpha,k\in\RR$, then 
$Pu \in \mathcal{A}^I_{\rm{phg}}(M,E)$ implies that $u\in \mathcal{A}^J_{\rm{phg}}(M,E)$,
where $J = I \U K$ for some index set $K>\alpha$ determined by $\Specb(P)$.
\end{theorem}
In particular, if $u$ has only Sobolev regularity, but is mapped by a b-differential operator
to a section with an asymptotic expansion at $\del M$, for instance if $u$ is in
the kernel of $P$, then $u$ must also have an expansion at $\del M$.

\medskip

For simplicity, the theorems in this section have been formulated for compact manifolds with boundary
such as our $\phi$-manifold $M$. However, it is easy to extend them to non-compact manifolds with compact boundary such as 
$V=B \times [0,\epsilon)$, which we will consider later, under suitable support assumptions.  Specifically, Theorem \ref{thm:bdd cpct b} holds if the Schwartz kernel of $P$ is compactly supported, and Theorem \ref{th:b-phgreg} holds for compactly supported sections. For Definition \ref{def:Spec} and Theorem \ref{thm: b-parametrix} it is sufficient to require that $P$ is b-elliptic (that is, ${}^b\sigma_m(P)$ is invertible) near the boundary, then of course the parametrix will only be valid near the boundary.

\subsection{Mapping and regularity in the $\phi$-calculus}\label{subsec:phi calc}

The small $\phi$-calculus, which contains parametrices for fully elliptic $\phi$-operators, is simpler than the b-calculus, since one does not need to
worry about complicated spectral or index sets. The price for this is the strong requirement of full ellipticity. Since the Gauss-Bonnet and Hodge Lapace operators are not fully elliptic, we need a bigger calculus, called the full $\phi$-calculus. Operators in the full $\phi$-calculus are again degenerate
operators on a compact manifold with boundary, but now the boundary is assumed to be the total space of a
fibre bundle $F \rightarrow \del M \stackrel{\phi}{\rightarrow} B$.  We always extend this fibration to a neighborhood of the boundary.   We fix an order of
degeneracy $a \in \mathbb{N}$.  Recall that the degeneracy of $\phi$-differential operators is described by the vector fields  \eqref{degvf}, where we always use coordinates $(x,y_1,\dots,y_b,z_1,\dots,z_f)$ adapted to the fibration. Sometimes we will write $\phi$-differential operators \eqref{eq:phidiffop} as
\begin{equation}
\label{eq:phidiffop2}
P=\sum_{j+|K| \leq m} P_{j, K}^{x,\bf y} \,(x^{1+a}\calD_x)^j
(x^{a}\calD_y)^K
\end{equation}
with $P_{j,K}^{x,\bf y}$ differential operators on the fibre $F_{x,\bf y}$, of order $\leq m-j-|K|$.

Once again, we study these operators by considering distributions that live on a blown up double space,
$M^2_\phi$, that has four boundary hypersurfaces, $\lf$, $\rf$, $\bff$ and $\ff$.  It is a degree $a$ blowup
of the b-double space, $M^2_b$, and $\ff$ is
the new boundary hypersurface created by this blowup. As for the b-calculus, we fix an index family $\calI=(\calI_\lf,\calI_\rf,\calI_\bff,\calI_\ff)$ for $M^2_\phi$ and define the full calculus $\Psi_\phi^{m,\calI}(M,\Omega^{1/2}_b)$ as the space of distributions on $M^2_b$ which are conormal with respect to the lifted diagonal of order $m$ and polyhomogeneous at the boundary with index family $\calI$. However, as is explained in \cite[Section 3.4]{GH1},  for the $\phi$-calculus it is natural to take these distributions valued in a half-density bundle denoted $\Omega_{\bphi}^{1/2}$ rather than in $\Omega_b^{1/2}$. Sections of $\Omega_{\bphi}$ behave like b-densities at $\lf,\rf,\bff$, but like $\phi$-densities at $\ff$, that is, they are $x^{-a(b+1)}$ times a b-density there, so that the total exponent of $x^{-1}$ is the same as the sum of all the exponents of $x$ in the vector fields \eqref{degvf}.

With this normalization, operators in $\Psi_\phi^{m,\calI}(M,\Omega^{1/2}_b)$ act naturally on b-half-densities on $M$. The extension to $\Psi_\phi^{m,\calI}(M,E)$ for a vector bundle $E$ over $M$ works as before, and the normalization also gives $\Diff_\phi^m(M,E) \subset \Psi_\phi^m(M,E) := \Psi_\phi^{m,(\emptyset, \emptyset, \emptyset, 0)}(M,E)$. 
The latter space is called the small $\phi$-calculus.

We first note mapping properties of elements of the full  $\phi$-calculus on Sobolev and 
polyhomogeneous spaces.

\begin{theorem}[Boundedness and compactness for $\phi$-operators]
\label{thm:bdd cpctphi}
Let $M$ be a $\phi$-manifold and let $E$ be a vector bundle over $M$. Let $P\in \Psi_\phi^{m,\calI}(M, E)$ and $\alpha,\beta,k\in\RR$.
\begin{enumerate}
 \item
  If 
\begin{equation}
 \label{eqn:condition bddness}
 \calI_\lf>\beta,\ \calI_\rf >-\alpha,\ \calI_\bff \geq \beta-\alpha,\ \calI_\ff \geq \beta-\alpha
  \end{equation}
and
\begin{equation}
 \label{eqn:>0 at corner}
 \text{strict inequality holds in \eqref{eqn:condition bddness} for at least one of $\bff,\ff$}
\end{equation}
then $P$ is bounded as an operator
\begin{equation}\label{eqn:sobolev map}
 x^\alpha H^{k+m}_\phi(M,E,\dvol_b)  \to x^\beta H^{k}_\phi (M,E,\dvol_b). 
\end{equation}
\item
If  $m<0$ and strict inequality holds everywhere in \eqref{eqn:condition bddness} then $P$, acting as in  \eqref{eqn:sobolev map}, is compact.
\end{enumerate}
\end{theorem}
Note that although the differentiations defining the $\phi$-Sobolev spaces are those arising from the $\phi$-vector fields \eqref{degvf}, the volume form in \eqref{eqn:sobolev map} is still the b-volume form. This is not surprising as this volume form relates to the conditions on the weights, so to the expansions at $\lf$ and $\rf$, and elements of the full $\phi$-calculus that vanish near $\ff$ are actually in the full b-calculus.
\begin{theorem}[Mapping of polyhomogeneous sections for $\phi$-operators]
Let $u \in \calA^I_{\rm{phg}}(M,E)$ and $P\in \Psi_{\phi}^{m,\calJ}(M,E)$.  Then
if $\calJ_\rf + I> 0$, we can define $Pu$ and we get that
$Pu \in \calA^K_{\rm{phg}}(M,E)$, where
$K = \calJ_\lf \U (\calJ_\bff+I)\U(\calJ_\ff+I)$.
\end{theorem}

We will also use the following result from \cite{GH1} that says how elements of the full $\phi$-calculus compose.   
\begin{theorem}[Composition in the full $\phi$-calculus]
\label{thm:composition phi}
If $P\in \Psi_{\phi}^{m,\calI}(M,E)$ and $Q \in \Psi_\phi^{m',\calJ}(M,E)$ and if $I_\rf+J_\lf>0$,
then $PQ \in \Psi_\phi^{m+m', \calK}(M,E)$, where, with $A=a(b+1)$,
\begin{align*}
\calK_\lf  &= \calI_\lf \U (\calI_{\bff} + \calJ_\lf ) \U (\calI_\ff +  \calJ_\lf ),\\
\calK_\rf  &= \calJ_\rf \U (\calI_\rf + \calJ_{\bff} ) \U (\calI_\rf + \calJ_\ff),\\
\calK_{\bff} &= (\calI_\lf + \calJ_\rf ) \U (\calI_{\bff} + \calJ_{\bff}) \U  (\calI_\ff + \calJ_{\bff}) \U (\calI_{\bff} + \calJ_\ff),\\
\calK_\ff &= (\calI_\lf + \calJ_\rf + A) \U (\calI_{\bff} + \calJ_{\bff}  + A) \U (\calI_\ff + \calJ_\ff).
\end{align*}
In particular, if one factor is in the small $\phi$-calculus and the other factor has index family $\calI$ then the composition also has index family $\calI$.
\end{theorem}

We now discuss some results in the small $\phi$-calculus.
To an operator $P \in \Psi_\phi^*(M,E)$ we may associate a $\phi$-principal symbol.
In particular, when $P \in \mbox{Diff}^m_\phi(M,E)$ is given by the sum in equation (\ref{eq:phidiffop}),
the $\phi$-principal symbol is given by:
\begin{equation}
\label{eq:phiprinc}
{}^\phi\sigma_m(P)(x,\bfy,\bfz,\tau,\eta,\zeta)=\sum_{j+|K|+|L| = m} a_{j, K,L}(x,{\bf y}, {\bf z}) \tau^j
\eta^K \zeta^L.
\end{equation}
We say that $P$ is $\phi$-elliptic if its $\phi$-principal symbol is invertible for $(\tau, \eta, \zeta) \neq 0$.

As in the b-calculus we need a second symbol to fully characterize Fredholm and regularity properties of $\phi$-operators. This is the normal operator, which for a $\phi$-differential operator $P$ as in \eqref{eq:phidiffop2} is given by
$$ N(P) =  \sum_{j+|K| \leq m} P_{j, K}^{0,\bf y} \,\calD_T^j \calD_Y^K. $$
This is a family of operators parametrized by $\bfy\in B$, acting on sections of the bundle $E_{|\partial M}$ pulled back to $\RR_T\times \RR_Y^b \times F_\bfy$ which are rapidly decaying as $|(T,Y)|\to\infty$. More generally, for operators $P \in \Psi_\phi^m(M,E)$ the normal operator may be defined using the restriction of the Schwartz kernel to $\ff$. The  {\em normal family} carries the same information as the normal operator. It is a family of operators on $F_\bfy$ parametrized by $(\tau,\eta)\in\RR\times\RR^{b}$ and by $\bfy\in B$. For $\phi$-differential operators it is 
 obtained by replacing $\calD_T$ by $\tau$ and $\calD_Y$ by $\eta$ in the normal operator, so
\begin{equation}
\label{eq:phinormal}
\Nhat(P)(\bfy,\tau,\eta) =
\sum_{j+|K| \leq m} P_{j, K}^{0,\bf y} \, \tau^j
\eta^K\,.
\end{equation}
 For computations the following characterization is useful (see  \cite[Equation (29)]{GH1}): Fix $\tau\in\RR,\eta\in\RR^b$ and $\bfy_0\in B$, in coordinates $\bfy$. Let $g(x,\bfy)=-\tau a^{-1}x^{-a} + \eta (\bfy-\bfy_0) x^{-a}$.
Then for a section $u\in\Cinfty(F_{\bfy_0},E)$ we have
\begin{equation}
\label{eqn:normal family osc}
\Nhat(P)(\bfy,\tau,\eta)u = \left(e^{-ig}P e^{ig}u \right)_{|\bfy=\bfy_0,x=0} 
\end{equation}

If $P$ is $\phi$-elliptic and $N(P)$ is invertible (which is equivalent to invertibility of $\Nhat(P)(\bfy,\tau,\eta)$ for each $\bfy,\tau,\eta$), then we say $P$ is {\em fully elliptic}. We have  (Theorem 9 from \cite{GH1}, but see also \cite{MaMe}):
\begin{theorem}[Parametrices for elliptic $\phi$-operators in the small $\phi$-calculus]
\label{thm: phi-parametrix}
Let $P \in \Psi^m_\phi(M,E)$ be a $\phi$-elliptic $\phi$-operator over a $\phi$-manifold, $M$.  Then there exists
an operator $Q \in \Psi_\phi^{-m}(M,E)$ such that $PQ= I + R_1$ and $QP = I + R_2$, where $R_i \in\Psi_\phi^{-\infty}(M,E)$. If $P$ is fully elliptic then $Q$ can be chosen so that 
$R_i \in x^\infty\Psi_\phi^{-\infty}(M,E)$.
\end{theorem}

This implies the following Fredholm and regularity result.
\begin{theorem}[Fredholmness and regularity of fully elliptic $\phi$-operators]
\label{thm:phireg}
Let $P\in\Psi^m_\phi(M,E)$ be a fully elliptic $\phi$-operator on a $\phi$-manifold.
Then $P$ is Fredholm
as a map $P:x^cH^k_\phi(M,E,\dvol_b) \rightarrow x^{c}H^{k-m}_\phi(M,E,\dvol_b)$ for any  $c,k\in\RR$.  
If $u$ is a tempered distribution and 
$Pu \in x^\alpha H^k_\phi(M,E,\dvol_b)$, 
then $u\in x^\alpha H^{k+m}_\phi(M,E,\dvol_b)$.  If $u$ is in any weighted Sobolev space and $Pu \in \calA^{J}_\phg(M,E)$
then in fact $u \in \calA^{J}_\phg(M,E)$.  
\end{theorem}


\section{Structure of the Laplacian}
\label{sec:laplacian}
Assume that $M$ is a $\phi$-manifold with boundary $\del M=Y\stackrel{\phi}\to B$ and $\phi$-cusp metric $g$.
In a neighborhood of the boundary we can extend the boundary fibration map, $\phi$,
to the interior using the product structure:
\begin{equation}
\label{eq:UV}
 U \cong Y \times [0,\epsilon) \stackrel{\Phi=\phi\times \id} {\longrightarrow}
V\cong B \times [0,\epsilon)  \stackrel{\Phi_0}\to [0,\epsilon).
\end{equation}
In order to understand the structure of the Gauss-Bonnet and Hodge Laplace
operators it is useful to recall a few facts about Riemannian fibrations.


First, let us recall which structures are associated to a fibration $\phi:Y\to B$.
The tangent bundle $TY$ has the tangent spaces to the fibres as natural subbundle, called the vertical tangent subbundle. The cotangent bundle $T^*Y$ has the forms annihilating vertical vectors as natural subbundle, called the horizontal cotangent subbundle. 
Let  ${\bf y} = (y_1, \ldots, y_{b})$ be  local coordinates on $B$ and  lift them to a partial set of 
local coordinate functions on $Y$. Then the exterior derivatives $dy_1, \ldots, dy_{b}$ span the horizontal cotangent subbundle. 
Now supplement these coordinate functions to a complete set of local coordinates for $Y$ and call the additional coordinates ${\bf z}=(z_1, \ldots, z_{f})$. We may think of the $z_j$ as fibre local 
coordinates. Choosing the $z_j$ is equivalent to choosing a trivialization of the fibration $\phi$ locally on $Y$. The set of coordinates $y_1,\dots,y_b,z_1,\dots,z_f$ defines vector fields $
\del_{y_1}, \ldots, \del_{y_b}, \del_{z_1}, \ldots, \del_{z_f}
$.
 The vertical tangent subbundle is locally spanned by the vector fields $\del_{z_1}, \ldots, \del_{z_f}$. 
 
Now consider the additional structures defined by Riemannian metrics $g_Y$ and $g_B$ on $Y$ and $B$, respectively,  for which $\phi$ is a Riemannian fibration.  The metric $g_Y$ defines the horizontal tangent subbundle (or horizontal distribution) of $TY$, as the orthogonal complement of the vertical tangent subbundle. For any $p\in Y$ the differential of the fibration $\phi:Y\to B$ restricts to an isomorphism of the horizontal subspace of $T_pY$ to  $T_{\phi(p)}B$, and the assumption that $\phi$ is a Riemannian fibration for the metrics $g_Y,g_B$ means that this isomorphism is an isometry for each $p$. Equivalently, $g_Y=\phi^*g_B + h$ where the symmetric two-tensor $h$ vanishes on the horizontal subspace. Dually, the dual metric on $T^*Y$ defines the vertical cotangent subbundle as orthogonal complement of the horizontal cotangent subbundle, and the natural map $T^*Y\to T^*F$, given by restricting a form to $TF$, restricts to an isomorphism of the vertical cotangent space and $T^*F$. If $\phi$ is Riemannian then this isomorphism is an isometry.

In coordinates, the horizontal tangent subspace will usually {\em not} be spanned by the coordinate vector fields $
\del_{y_1}, \ldots, \del_{y_b}$. 
A choice of coordinates $\bfy, \bfz$  for which this does happen is  possible if and only if the horizontal distribution is integrable, that is, if
there is a local foliation by
submanifolds that are orthogonal and transversal to the fibers. By Frobenius' theorem this is equivalent to the vanishing of the curvature, denoted $R$, of the horizontal distribution.   Note that in this case,
the requirement that $\phi$ is a Riemannian fibration means 
that the horizontal submanifolds are also totally geodesic.  
In contrast, the fibres of $\phi$ will not generally be totally geodesic.  This happens if and only if their 
second fundamental forms, denoted $\mathbb{II}$, vanish.  

In coordinates $\bfy,\bfz$ the matrix of $g_Y$ has the form
\begin{equation}\label{eqn:gY matrix}
g_Y = \left(
\begin{array}{lll}
 A_{00}({\bf y}) &  A_{01}({\bf y},{\bf z})  
                                                                                                          \\
A_{10}({\bf y},{\bf z}) & A_{11}({\bf y},{\bf z})\\
\end{array}
\right)
\end{equation}
where $A_{00}$ only depends on $\bfy$ because $\phi$ is Riemannian. The submatrix $A_{00}$ encodes $g_B$, and $A_{01},A_{10},A_{11}$ encode $h$.
If the curvature of the horizontal distribution vanishes then coordinates can be chosen so that the $A_{01}$ and $A_{10}$ terms vanish, and if in addition the second fundamental form of the fibres vanishes then
the $A_{11}$ term is a function only of the $\bfz$ (fibre) variables.  So the vanishing of 
both terms means that the Riemannian fibration is locally trivial in a metric sense.

Now consider the metric $g$ in \eqref{eq1}, where we fix a product decomposition \eqref{eq:UV}. Choose local coordinates $\bfy,\bfz$ associated to $\phi$. We use the above discussion with $h$ replaced by $x^{2a}h$. This would introduce factors $x^{2a}$ in front of the $A_{01},A_{10},A_{11}$ terms of \eqref{eqn:gY matrix}. However, it is advantageous to write $g$ in terms of  the  rescaled local
coordinate vector fields
\begin{equation}
 \label{eqn:cphi vector fields}
x\del_x, \{\del_{y_i}\}_{i=1}^b, \{x^{-a}\del_{z_j}\}_{j=1}^{f}
\end{equation}
then the metric $g$
has the form
\begin{equation}\label{eqn:g matrix}
g = \left(
\begin{array}{lll}
1 & 0 & 0 \\
0 & A_{00}({\bf y}) &  x^{a}A_{01}({\bf y},{\bf z})  
                                                                                                          \\
0 & x^{a}A_{10}({\bf y},{\bf z}) & A_{11}({\bf y},{\bf z})\\
\end{array}
\right).
\end{equation}
 
The fact that this is smooth and non-degenerate up to $x=0$ motivates the introduction of the vector bundle $\cphiT M$ over $M$, defined by the requirements that the vector fields \eqref{eqn:cphi vector fields} form a local basis of sections of $\cphiT M$ over any coordinate patch in $U$ and that sections of $\cphiT M$ over $\interior M$ are smooth vector fields on $\interior M$. There is a canonical identification of $\cphiT M$ and $TM$ over $\interior M$, but not over $M$, since $x^{-a}\partial_{z_j}$ is not defined as section of $TM$ at $x=0$ and $x\partial_x$ vanishes there.
 Now \eqref{eqn:g matrix} shows that $g$, which was defined only in the interior $x>0$, defines a metric on the bundle $\cphiT M$ over all of $M$, i.e. in $x\geq 0$. 
The dual bundle $\cphiT^*M$ is locally spanned by $\frac{dx}x, dy_1,\dots,dy_b, x^a dz_1,\dots,x^a dz_f$, over $U$. 

Similarly, over $V=B\times[0,\eps)$ we have the bundle $\bT V$ locally spanned by $x\partial_x, \partial_{y_1},\dots,\partial_{y_b}$. The differential of $\Phi$ maps $\cphiT U\to\bT V$ by sending $x^{-a}\partial_{z_j}$ to zero.
From $V=B\times[0,\eps)$ we have $\bT^*V = \Span\{\tfrac{dx}x\} \oplus T^*B$.

We apply the discussion above about Riemannian fibrations to the fibration $\Phi:U\to V$ and the bundles  $\cphiT^*U$, $\bT^*V$.
The inclusion $TF\subset TU$ is replaced by $x^{-a}TF \subset \cphiT U$ with the dual restriction map $\cphiT^*U\to x^a T^*F$.
We obtain the orthogonal decomposition of vector bundles
\begin{equation}
\label{eqn:cphiT*decomp} 
 \cphiT^*U = \Span\{\tfrac{dx}x\} \oplus \calH \oplus \calV
\end{equation}
Here $\Span\{\tfrac{dx}x\} \oplus \calH$ is the image of $\Phi^*:\bT^*V\to \cphiT^*U$ and $\calV$ is its orthogonal complement with respect to the dual of the metric \eqref{eqn:g matrix}. The map $\cphiT^*U\to x^a T^*F$  restricts to an isomorphism $\calV\to x^a T^*F$, which is an isometry since $\Phi$ is Riemannian.
Under the identification of $\cphiT^* U$ and $T^*U$ over the interior, \eqref{eqn:cphiT*decomp} is the horizontal/vertical decomposition with respect to the singular metric $g$; the use of the rescaled bundles allows us to extend this decomposition smoothly to the boundary.

The Gauss-Bonnet operator $D_M$ and Hodge Laplacian $\Delta_M$ for the metric $g$ are first defined over $\interior M$ only, but when considered as operators acting on sections of the bundle $E=\Lambda\cphiT^*M$ they are elliptic $\phi$-cusp-operators on $M$, that is, 
$$D_M = x^{-a}P, \ P\in \Diff_\phi^1(M,\Lambda\cphiT^*M), \quad
\Delta_M = x^{-2a}T, \ T\in \Diff_\phi^2(M,\Lambda\cphiT^*M)$$
with $P,T$ $\phi$-elliptic.
This follows from general principles as shown in \cite{Me-icm} or from concrete calculation, see \cite{HHM}.
Concretely, this means that if we write forms on $M$ near the boundary as
\begin{equation}
\label{eqn:rescaled form} 
 \sum_{K,L} \left(a_{K,L} + \tfrac{dx}x\wedge b_{K,L}\right),
 \quad a_{K,L} = \alpha_{K,L}\, d\bfy^K \wedge (x^a d\bfz)^L,
 \quad b_{K,L} = \beta_{K,L}\,  d\bfy^K \wedge (x^a d\bfz)^L 
\end{equation}
for multiindices $K,L$
and then consider how $D_M$, $\Delta_M$ act on the coefficients $\alpha_{K,L}$, $\beta_{K,L}$, then $P,T$ are expressible in terms of the $\phi$-vector fields \eqref{degvf}, with coefficients smooth up to the boundary. 

In addition, we can write $D_M$ and $\Delta_M$ in terms of the vertical/horizontal decomposition, and this will be essential for our analysis. 
The decomposition \eqref{eqn:cphiT*decomp} of $\cphiT^*U$ induces a decomposition of $\Lambda\cphiT^* U$.
It is useful to write the exterior derivative $d_U$ on $U$ in terms of this orthogonal decomposition, rather than in coordinates $x,\bfy,\bfz$, since then its adjoint is easy to compute. We obtain (see \cite{HHM}):
\begin{equation}
\label{eqn:d_U}
\begin{aligned}
d_U &= x^{-a}d_{F} \\
&+ d_B \mid_{\calV} + \Phi^*d_B - \mathbb{II}\\
&+ x^{a}R \\
&+ (\Phi_0\circ\Phi)^*(d_x)-A. 
\end{aligned}
\end{equation}
In this decomposition, the term $d_F$ is the exterior derivative on the fibres, under the isomorphism $\calV\to x^a T^*F$. The term $d_B \mid_{\calV}$ represents the action of the 
derivatives in the $B$ directions on the  $\calV$-components of a form and the term
$\Phi^*{d_B}$ is the pullback to $U$ of the $B$ differential on $V=B\times[0,\eps)$. 
Here we identify $\calH$ with $T^*B$ via the differential of $\Phi$. The terms $\mathbb{II}$ and $R$ are the second fundamental form and curvature operators for the metric $g$. $R$ is independent of $x$ and $\II$ depends smoothly on $x$.  
The operator $d_x$ is pulled back from $[0,\eps)$ via $\Phi_0\circ\Phi:U\to [0,\eps)$ in \eqref{eq:UV}. It first acts by $x\del_x$ then
wedges with $\frac{dx}x$, so it is a b-operator.
The 0th order differential operator $A$
acts on each summand in \eqref{eqn:rescaled form} by
\begin{equation}
 \label{eqn: def A}
A(a_{K,L} + \tfrac{dx}x \wedge b_{K,L})
= a |L| \,\tfrac{dx}x \wedge a_{K,L} 
\end{equation}

Note that in terms of the decomposition of forms by subbundle degrees, we get that
$d_{F}$  increases the $\calV$ degree by 1,
the various $d_B$ and $\mathbb{II}$ terms increase the $\calH$ degree by 1
and $R$ increases the $\calH$ degree by 2 and decreases
the $\calV$ degree by 1.   
We define the following combination to simplify notation:
$$
d_{B} - \mathbb{II} := d_B \mid_{\calV} + \Phi^*d_B - \mathbb{II}.
$$

Passing to the Gauss-Bonnet operator, we fix the following notation:
\begin{gather}
D_U= d_U + d_U^*,\ D_B = (d_{B} - \mathbb{II}) + (d_{B} - \mathbb{II})^*,\ D_F = d_F + d^*_F, \label{eqn:DU DB DF}\\
        {\bf R} =R + R^*,\ xD_x = (\Phi_0\circ\Phi)^*d_x + ((\Phi_0\circ\Phi)^*d_x)^*,\ \bfA = A+A^* \notag
\end{gather}
Then we get
\begin{equation} \label{eqn:D_U final}
 D_U = x^{-a} D_F + D_B + x^a \bfR + xD_x - \bfA
\end{equation}

\subsection{The ``split" property for the Gauss-Bonnet operator and the Hodge Laplacian}
In this section, we identify some of the properties of the Gauss-Bonnet operator, $D_U$, and the 
Hodge Laplacian, $\Delta_U$, that will permit us to construct parametrices for them using a combination 
of b- and $\phi$-calculus operators, and 
that form the model for more general ``split elliptic" operators in \cite{GH3}.
When multiplied with $x^a$ and $x^{2a}$, respectively, both operators are elliptic as $\phi$-operators in the standard sense, 
so we will focus on additional properties they have in the neighbourhood $U$ of the boundary.  

We calculate the normal operator of $x^a D_U$. First, we consider the normal family of $x^a d_U$.
Recall
that the normal family of a $\phi$-operator in $\Diff_\phi^m(M,\Lambda\cphiT^*M)$ is a family of operators at the boundary
acting on $C^\infty(F_\bfy;\Lambda \cphiT^*M)$ and parametrized by $({\bf y},\tau,\eta) 
\in B \times \mathbb{R}^{b+1}$. 
Consider the bundle $\cphiT^*M$ at a boundary point $(0,\bfy,\bfz)$. Since the metric \eqref{eqn:g matrix} is block diagonal at $x=0$, the decomposition \eqref{eqn:cphiT*decomp} reduces to $\cphiT_{(0,\bfy,\bfz)}^*M = W_\bfy^* \oplus x^a T_\bfz^*F_\bfy$ where
$W_\bfy = \bT_{(0,\bfy)}V$. Since $W_\bfy$ is a vector space, we have $W_\bfy^*=T_w^*W_\bfy$ for any $w\in W_\bfy$. If we identify $x^a T_\bfz^*F_\bfy$ with $T_\bfz^*F_\bfy$ and change the order, then we get a natural identification, for any $w\in W_\bfy$,
$$ \cphiT_{(0,\bfy,\bfz)}^*M = T_{(\bfz,w)}^*(F_\bfy\times W_\bfy)
.$$
When calculating $\Nhat(x^a d_U)$ from \eqref{eqn:d_U}, the zero order terms $\II$, $R$, $A$ drop out. Using \eqref{eqn:normal family osc} one easily computes $\Nhat(x^a d_U)(\bfy,\tau,\eta)u = d_Fu + i(\tau \tfrac{dx}x + \eta\,d\bfy)\wedge u$. When interpreting $\tau,\eta$ as coordinates on $\bT^*_{(0,\bfy)}V=W_\bfy^*$, the second term is simply the symbol of the exterior derivative operator $d_{W_\bfy}$ on the vector space $W_\bfy$. Taking adjoints, where the scalar product on $W_\bfy$ is given by $g_\bfy:=\frac{dx^2}{x^2}+g_B(\bfy)$, we get
\begin{equation}\label{eqn:normal operator}
  N(x^a D_U)(\bfy) = D_{F_\bfy} + D_{W_\bfy}  
\end{equation}
acting on $\Lambda T^*(F_\bfy\times W_\bfy)$, 
which is the Gauss-Bonnet operator on $F_\bfy\times W_\bfy$ with metric $h_{|F_\bfy} + g_\bfy$. Note that this shows that the space $\RR^{b+1}$ in the local definition of the normal family is, invariantly, $W_\bfy=\bT_{(0,\bfy)}V$.

Define the vector bundle
\begin{equation}
\label{eqn:def K} 
 \calK = \text{the bundle of $F$-harmonic forms over }V.
\end{equation}
More precisely, for any $\bfy\in B$ let 
$ \tilde\calK_\bfy = \ker D_F\subset \Cinfty(F_\bfy, \Lambda x^a T^*F_\bfy)$ be the space of harmonic forms on $F_\bfy$ for the metric $g_{|F_\bfy}=x^{2a}h_{|F_\bfy}$, and let 
$ \calK_{(x,\bfy)} = (\Lambda \bT^*_{(x,\bfy)}V) \otimes 
\tilde\calK_\bfy
\subset \Cinfty(F_{(x,\bfy)},\Lambda\cphiT^*U)$, using the isometric identification of $x^a T^*F$ with $\calV$. By the Hodge theorem, $\dim \tilde\calK_\bfy$ is finite and independent of $\bfy$, so $\calK$
is a vector bundle over $V$. 
Note that we keep the $x$ factors in the bundle. In Section \ref{subsec:proof main} we explain how $\calK$ relates to the bundle $\calKbar$ mentioned in the introduction. It is important to distinguish these bundles carefully in order to get the weight conditions in the main theorems right.
We also define the infinite dimensional vector bundle
\begin{equation}\label{C}
\mathcal{C} = \mathcal{K}^\perp = \mbox{forms on }F\ \mbox{that are perpendicular to the fibre harmonic forms},
\end{equation}
with respect to the $L^2$ scalar product on the fibres.
We have $\Cinfty(U,\Lambda\cphiT^*M) = \Cinfty(V,\calK) \oplus \Cinfty(V,\calC)$.
We refer to sections of $\mathcal{C}$ as ``fibre-perpendicular forms" over $V$.

Now we can define the following families of operators pointwise over $V$:
\begin{eqnarray}
\label{eq:calK}
\Pi &=& \mbox{ orthogonal projection onto } \mathcal{K}\\
\Pi^\perp:=I-\Pi&=& \mbox{ orthogonal projection onto }\mathcal{C}.
\end{eqnarray}
Observe that the family  $\Pi(x,{\bf y})$ of projection operators is smooth in the variable $y$ 
since the family of metrics on $F$ is smooth, and constant in  $x$, since these metrics
are independent of $x$.
These families of projection operators on $C^\infty(F,\Lambda \cphiT^*U)$ together define operators, also denoted $\Pi$ and $\Piperp$, on $C^\infty(U,\Lambda \cphiT^*U)$.

Note that the normal operator
$N(x^aD_M)=D_F+D_W$ commutes with $\Pi$. In particular, it maps   
$\Cinfty(V,\calC)\to\Cinfty(V,\calC)$. Its restriction to this subspace is invertible since its square is $\Delta_F + \Delta_W\geq \min\limits_{\bfy\in B}\lambda_\bfy>0$, where $\lambda_\bfy$ is the smallest positive eigenvalue of $\Delta_{F_\bfy}$.
We can decompose the restriction $D_U$ of $D_M$ to $U$ using these projections as the sum of four pieces:
$$
(D_U)_{00}:= \Pi D_U\,\Pi, \qquad (D_U)_{01}:= \Pi D_U\,\Pi^\perp,
$$
$$
(D_U)_{10}:= \Pi^\perp D_U\,\Pi, \qquad (D_U)_{11}:= \Pi^\perp D_U\,\Pi^\perp.
$$
In order to construct the parametrix for $D_M$, we need to study the properties of each of these
four pieces.  It is convenient to organise these in a matrix:
\begin{equation} \label{eqn:D_U}
D_{U} =
\left(
\begin{array}{ll}
(D_U)_{00} & (D_U)_{01} \\
(D_U)_{10} & (D_U)_{11}
\end{array}
\right)
=: x^{-a}
\left(
\begin{array}{ll}
x^a P_{00} & x^a P_{01} \\
x^a P_{10} & P_{11}
\end{array}
\right) .
\end{equation}

In \cite{HHM}, it is proved that the operator $\frakd:=\Pi(d_B-\mathbb{II})\Pi$
is the differential on forms over $B$ with values in $\calK$, with respect to the flat connection on $\calK$ induced by the metric $g$.  When we project 
$D_U$ using $\Pi$ we obtain from \eqref{eqn:D_U final}:
\begin{equation}\label{eqn:P_00}
\Pi D_U \Pi = P_{00} = \frakd + \frakd^* + xD_x - \bfA
+x^a \Pi {\bf R} \Pi,
\end{equation}
which is  a b-operator acting on sections of $\calK$. The fact that $\frakd$ is the differential on $\calK$-valued differential forms over $B$ implies that $P_{00}$ is b-elliptic.  
The operators $P_{01}$, $P_{10}$ and $P_{11}$ have smooth coefficients up to $x=0$ and are  $\phi$-operators in a sense made precise in the next
section.

The Hodge Laplacian is the square of the Gauss-Bonnet operator, so we can understand
it also in terms of  its $\calK$ and $\calC$ components.  Using \eqref{eqn:D_U} we obtain
\[
\Delta_{U} = x^{-2a}
\left(
\begin{array}{ll}
x^{2a}(P_{00})^2+  x^{2a}P_{01}P_{10}& x^{2a} P_{01}x^{-a}P_{11} + x^{2a} P_{00} P_{01}\\
  x^aP_{11}P_{10}+x^{2a}P_{10}P_{00}& x^aP_{11}x^{-a}P_{11}+x^{2a}P_{10}P_{01}
\end{array}
\right)
\]
\begin{equation} \label{eqn:Delta_M}
=:
x^{-2a}\left(
\begin{array}{ll}
x^{2a}T_{00}& x^a T_{01}^\prime + x^{2a} T_{01}\\
  x^a T_{10}^\prime + x^{2a} T_{10}& T_{11}
\end{array}
\right).
\end{equation}
Squaring \eqref{eqn:D_U final} we see that $T_{00} = \Pi D_U^2 \Pi$  is a second order b-operator acting on sections of $\calK$, which is 
 $b$-elliptic.  
The term $x^{2a}P_{10}P_{01}$ as well as $x^a P_{11}x^{-a} -P_{11}$ vanish at the boundary, so the normal family
of the lower right term is the same as the normal family of $P_{11}^2$.

The critical difference between analysis of the Gauss-Bonnet and Hodge Laplace operators
comes from the fact that the off-diagonal terms in $x^{2a}\Delta_{M}$ generally vanish like $x^{a}$.
It turns out that our parametrix construction requires the greater order of vanishing
$x^{2a}$, so this introduces a restriction on the metrics we consider in this paper.  
Since $xD_x$ and $\bfA$ commute with $\Pi$, we have from \eqref{eqn:D_U final}
$$
P_{01}=\Pi \left[D_B + x^{a}\bfR\right]\Pi^\perp,\quad
P_{10}=\Pi^\perp \left[D_B + x^{a}\bfR\right]\Pi,
$$ 
so we see that the problematic term in each one is the one involving $D_{B}$. 
The term vanishes if this operator commutes with $\Pi$. Also note that, in this case, we
have $T_{00} = (P_{00})^2 + O(x^{2a})$. 
The commutation condition is satisfied in the setting
of $\QQ$-rank 1 locally symmetric spaces.  However, we are interested in understanding
how to study manifolds also that generalise symmetric spaces.  This question of how better
to characterise boundary fibre bundles with this property was studied by J. Mueller
in \cite{JM}, and he is taking it up again together with the authors of this paper
in current work.


\section{The extended $\phi$-calculus}
\label{sec:ext phi calculus}
In this section we consider
how to lift fiber-harmonic b-kernels to the $\phi$-double space.  The motivation
for this is as follows.
Philosophically, we will create a parametrix for the Gauss-Bonnet and Hodge
Laplace operators over a $\phi$ manifold by using the $\phi$ parametrix for the fiber-perpendicular
part of the operator and the b-parametrix for the fiber-harmonic part of the operator near the boundary,
and the standard parametrix on the interior.  If the off-diagonal terms of our operators with respect
to the fiber-harmonic, fiber-perpendicular splitting vanish in a neighborhood of the boundary, then
in fact this simple diagonal parametrix suffices as a parametrix for the operator.  However, when 
there are off-diagonal terms, we will need to improve this parametrix.  In order to do this, and 
to understand the result we obtain from this, we need to understand how the b-parametrix composes
with various parts of our $\phi$-operator. We do this by reinterpreting (lifting) the b-parametrix as a $\phi$-operator and then using the composition theorem for $\phi$-operators, Theorem \ref{thm:composition phi}. Since the fibration $\phi$ is only defined near the boundary, these considerations take place on the neighborhood $U\equiv \partial M\times
[0,\eps)$ and its projection $V=B\times[0,\eps)$. 
  
In order to analyze the lift of b-operators on $V$ to $\phi$-operators on $U$, let us recall the spaces on which the Schwartz kernels of these operators live.
The boundary fibration determines the blown-up double space where the Schwartz kernels
of $\phi$-operators on $U$ live.  We denote this space, constructed in \cite{GH1}, by $U^2_\phi$. It is constructed as an $a$-quasihomogeneous blow-up from the b-double space $U^2_b$. 
We also have the b-double space $V^2_b$, and we can do the corresponding $a$-quasihomogeneous blow-up of it, with respect to the
 trivial fibration $B\to B$ of the boundary of $V$, to obtain $V^2_\phi$.
It is useful to put the various double spaces in a commutative diagram:
\begin{equation}
\label{eq:double spaces diagram}
\xymatrix{
U^2_\phi \ar@{->}[r]^{\tilde{\tilde\Phi}^2}\ar[d]^{\beta_{\phi,Y}} & V^2_\phi \ar[d]^{\beta_{\phi,B}}\\
U^2_b \ar@{->}[r]^{\tilde\Phi^2} \ar[d]^{\beta_{b,Y}}& V^2_b \ar[d]^{\beta_{b,B}}\\
U^2 \ar@{->}[r]^{\Phi^2} & V^2\\
}
\end{equation}
The rows in this diagram represent fiber bundles with fibres $F^2$.
The columns are sequences of blow-down maps.
The diagram embodies the fact that in both the b- and $\phi$-blowups of the $U$ double space,
the fibres $F^2$ are carried along like parameters.
In the parametrix construction we construct a local parametrix for the fiber harmonic part of the operator
on $V^2_b$ (with coefficients in the harmonic form bundle).  
To combine this with the other pieces of the parametrix, we
need to lift this vertically under $\beta_{\phi,B}$, then horizontally by $\tilde{\tilde{\Phi}}^2$ (the other order would work the same, of course), to get a kernel on 
$U^2_\phi$. For notational simplicity in the rest of this section we will suppress the bundle coefficients $E$.
\medskip

We first consider the lift of kernels under the vertical map $\beta_{\phi,B}$ in \eqref{eq:double spaces diagram}.

\begin{proposition}\label{prop:lift}
If $T \in \Psi_b^{m, \calI}(V;\calK)$ then its kernel, $K_T$, lifts by the map
$\beta_{\phi,B}$ to give the kernel
of an operator $\tilde{T} \in \Psi_\phi^{m, \calJ}(V,\calK) + \Psi_\phi^{-\infty,\calJ'}(V,\calK),$
where $\calJ$, $\calJ'$ agree with $\calI$ at $\lf,\rf,\bff$ and
\begin{equation}
\label{eq:lift index sets b phi}
\calJ_{\ff} = \calI_{\bff} + a(-m), \qquad \calJ'_{\ff} = \calI_{\bff} + a((-m)\overline{\cup}(b+1) ),
\end{equation}
where $(-m)$ denotes the index set $(-m+\NN_0)\times \{0\}$ for any $m\in\RR$ and $aJ = \{(az,k):\, (z,k)\in J\}$.
\end{proposition}


The proof of this proposition will be given in \cite{GH3}.
However, for intuition we can make the following remarks.
\begin{enumerate}
\item Note that the shift by $-a m$ in the first term at the front face in \eqref{eq:lift index sets b phi} is easily explained in the case of differential operators. For example, $x^m \partial_x^m \in \Psi_b^m$, but only after multiplying it with $x^{a m}$ does it become an operator in $\Psi_\phi^m$.
However, we will only use the proposition for lifting the parametrix, i.e. for negative $m$.
\item The second summand of $\tilde{T}$ may have log terms in its expansion at the front face even if $T$ doesn't (encoded by the sign $\overline{\cup}$ in the term $(-m)\overline{\cup}(b+1)$ in \eqref{eq:lift index sets b phi}); these come from log terms in the expansion of the kernel of $T$ at the diagonal.
\end{enumerate}
We now consider the lift of a kernel under the horizontal map $\tilde{\tilde\Phi}^2$ in \eqref{eq:double spaces diagram}, more precisely, the lift of the kernel of an operator in $\Psi^m_\phi(V,\calK)$ to a kernel on $U^2_\phi$. 
Consider kernels of operators in $\Psi^m_{\phi}(U)$.  These are distributions on $U^2_\phi$ with a conormal singularity along the lifted diagonal.  We can think of these
as kernels of the form  $K(x,x',y,y',z,z')$ which are conormal to the diagonal $x=x',y=y',z=z'$
uniformly (in an appropriate rescaled sense given by the blowup) down to $x=x'=0$.

By comparison, the kernel of an operator in $\Psi^m_\phi(V,\calK)$
is a finite sum of terms of the form
$K_B(x,x',y,y')\otimes L(z,z')$, where $K_B$ is a distribution on $V^2_\phi$ with a conormal singularity along the lifted diagonal.  Again, we can think of this as being conormal to 
the set $x=x',y=y'$ uniformly up to the boundary in the same rescaled sense as above.  The functions $L\in\calK\otimes\calK$ in this sum are smooth since harmonic forms on $F$ are smooth.  
To lift kernels of operators in  $\Psi^m_\phi(V,\calK)$ horizontally, we simply consider them as
distributions on $U^2_\phi$.
Note that these lifts are not the kernels of pseudodifferential operators.  This is because 
pseudodifferential operators are singular only along the diagonal, $x=x',y=y',z=z'$, whereas
these horizontally lifted distributions are 
singular on the larger set  $x=x',y=y'$, which we call the {\em fibre diagonal}.  
When we have lifted the kernel of an element in $\Psi^m_\phi(V,\calK)$, we will also generally
want to glue it to an interior kernel, so we will multiply the result by a smooth
cutoff function supported compactly in $U^2_\phi$ and equal to 1 near  $\bff$ and $\ff$.
We define a space of operators on $U$ in which these cutoff lifts live as follows.
\begin{definition}
 Let $\Psi^m_{\phi,F}(U)$ be the space of operators which are $\phi$-operators on $V$ valued in smoothing operators on $F$, and with kernels compactly supported in $U^2_\phi$. The kernels of operators in 
 $\Psi^m_{\phi,F}(U)$ are conormal of degree $m$ with respect to the fibre diagonal, smoothly up to the front face, and vanish to infinite order at the other faces.
\end{definition}

We make some notes about this space.
\begin{enumerate}
\item Note that this space does not require that the smoothing operators on $F$ vanish on fibre-perpendicular sections.  Thus
this space contains more than just the horizontal lifts of operators in $\Psi^m_\phi(V,\calK)$.
\item We may generalise this space in a natural way to include operators with polyhomogeneous
expansions given by an index family $\calG$ at the boundary hypersurfaces in $U^2_\phi$, which we will denote by $\Psi^{m,\calG}_{\phi,F}(U)$.
\item In Equation \eqref{eq:calK}, we defined the projection operator $\Pi$ and we used it in our decomposition of the Gauss-Bonnet operator.  Note that this is the lift of the identity operator in $\Psi^0_\phi(V,\calK)$
to an operator in $\Psi^0_{\phi,F}(U)$.  It is not a pseudifferential operator on $U$.
\end{enumerate}
Now consider the operator $\Pi^\perp$.  This is not smoothing in the $F$ factor, so it 
is not an element of $\Psi^0_{\phi,F}(U)$.  Rather, it is the difference
of the identity element in $\Psi^0_{\phi}(U)$ and $\Pi \in \Psi^0_{\phi,F}(U)$. 
Thus when we consider the compositions $\Pi x^a D_U\Pi$, $\Pi^\perp x^aD_U \Pi$ and 
$\Pi x^aD_U \Pi^\perp$, we will also arrive at such sums.  This motivates the following definition.
\begin{definition}
 Define the extended calculus as the sum of two pieces:
\begin{equation}\label{eq:full calc}
 \Psi^{m,\calG}_{\phi,\ext}(M) = \Psi^{m,\calG}_\phi (M) +\Psi^{m,\calG}_{\phi,F}(U).
 \end{equation} 
\end{definition}

As noted above, elements in this calculus need not be pseudodifferential operators since their Schwartz kernels may have singularities outside the diagonal. It should be noted, however, that this is a minor extension which only serves to formulate the parametrix construction in a simple way. The final parametrix that we consider is actually a pseudodifferential operator in the interior of $M$. This follows from ellipticity by the classicial pseudodifferential calculus.

There is no interior symbol map for operators in the extended calculus. However, the normal operator is still defined for elements in the subspace $ \Psi^m_{\phi,\ext}(M)$.
Recall that the normal operator of the $\phi$-calculus takes values in  the space  of suspended pseudodifferential $\phi$-operators, 
which we denote by $\Psi^m_{\sus-\phi}(\partial M)$. These are pseudodifferential operators on $\RR_T\times\RR_Y^b\times F$ (locally near a point of $B$) which are translation invariant in $T,Y$, so that they are given by a convolution kernel $K(T,Y,z,z')$ with a conormal  singularity at $T=0,Y=0,z=z'$, and such that $K$ decays rapidly as $(T,Y)\to\infty$. By adding terms that satisfy the same condition except that they have a conormal singularity on $T=0,Y=0$, we obtain the extended suspended calculus, denoted $\Psi^m_{\sus-\phi,\ext}(\partial M)$.

The following proposition shows that $\Psi^m_{\phi,\ext}(M)$, $\Psi^m_{\sus-\phi,\ext}(\partial M)$ behave much like the standard $\phi$-pseudodifferential calculus and suspended calculus, respectively. As before, we suppress the bundle $E$ in the notation.
\begin{proposition} \label{lem: extended phi-calc}
 \mbox{}
\begin{enumerate}
 \item $\Psi^*_{\phi,\ext}(M)$ and  $\Psi^*_{\sus-\phi,\ext}(\partial M)$ are  closed under composition.
 \item
 $ \Psi^m_{\phi,\ext}(M) \circ \Psi^{-\infty}_{\phi}(M) \subset \Psi^{-\infty}_{\phi}(M)$,
 $ \Psi^{-\infty}_{\phi}(M)\circ \Psi^m_{\phi,\ext}(M) \subset \Psi^{-\infty}_{\phi}(M)$
 and similar results hold for $\Psi^*_{\sus-\phi,\ext}$.
\item
 Let $P\in\Diff_\phi^m(M)$ be $\phi$-elliptic and such that $N(P)$ is diagonal with respect to the splitting $\Cinfty(\RR^{b+1}\times F) = \Cinfty(\RR^{b+1},\calK) \oplus \Cinfty(\RR^{b+1},\calC)$ and so that its $\Cinfty(\RR^{b+1},\calC)\to\Cinfty(\RR^{b+1},\calC)$ part is invertible. Let
\begin{equation} \label{eqn:N(P) inverse}
  B = 
\begin{cases}
 \left[N(P)_{|\Cinfty(\RR^{b+1},\calC)\to\Cinfty(\RR^{b+1},\calC)})\right]^{-1} & \text{ on }\Cinfty(\RR^{b+1},\calC)\\
 0 & \text{ on }\Cinfty(\RR^{b+1},\calK)
\end{cases}
\end{equation}
Then $B\in \Psi^{-m}_{\sus-\phi,\ext}(\partial M)$.
\item
The short exact  sequence for the normal operator of the $\phi$-calculus extends to a short exact sequence, for each $m$,
$$ 0 \longrightarrow x\Psi^m_{\phi,\ext}(M) \hookrightarrow \Psi^m_{\phi,\ext}(M)
\stackrel{N}{\longrightarrow} \Psi^m_{\sus-\phi,\ext}(\dM)
\longrightarrow 0.
$$
\item
Any $A\in\Psi^m_{\phi,\ext}(M)$ defines a bounded operator, for all $k$,
$$ H^{k+m}_\phi(M,\dvol_b) \to H^{k}_\phi(M,\dvol_b) $$
\end{enumerate}
\end{proposition}
\begin{proof}
 
\begin{enumerate}
 \item
 This follows essentially the same lines as the proof that $\Psi^*_\phi(M)$ is closed under composition, see Theorem 8 in \cite{GH1}.
 \item
 This follows from $\Psi_{\phi,\ext}^{-\infty}(M) = \Psi_\phi^{-\infty}(M)$.
 \item
 We show this by a variation on the well-known argument that the inverse of an invertible pseudodifferential operator is pseudodifferential again.  Since $P$ is $\phi$-elliptic, $N:=N(P)$ is an elliptic element of $\Psi^m_{\sus-\phi}(\partial M)$. Hence it has a parametrix $C\in \Psi^{-m}_{\sus-\phi}(\partial M)$ so that $NC=I+R$, $CN=I+R'$ with $R,R'\in \Psi^{-\infty}_{\sus-\phi}(\partial M)$. We multiply these identities from the left and right with $B$ respectively and solve for $B$. This gives $B=BNC-BR=BNC-(CNB- R'B)R$. Now use $NB=BN=\Pi^\perp=\Id-\Pi$ to obtain
 $B= C - \Pi C - S$ where $S=- C\Pi^\perp R + R'BR \in \Psi^{-\infty}_{\sus-\phi}(\partial M)$ by part 1 and standard facts, so $B\in -\Pi C + \Psi^{-m}_{\sus-\phi}(\partial M)$. 
The result follows.
\item
This is obvious; compare Lemma 4 in \cite{GH1}.
\item
 By composing with invertible fully elliptic $\phi$-operators one reduces to the case $k=m=0$. Since $\phi$-operators of order zero are bounded on $L^2$, we only need to check $L^2$-boundedness for $P\in\Psi^0_{\phi,F}(U)$. Using a partition of unity we may assume that the kernel of $P$ is supported near the fibre over a point of $B\times B$. Then by fixing $z,z'\in F$ we may regard $P$ as a family of $\phi$-operators on $V$ parametrized by $z,z'\in F$. Each of these is bounded on $L^2(V)$, uniformly in $z,z'$. The usual argument showing that Hilbert-Schmidt operators are bounded in $L^2$ then shows that  $P$ is bounded on $L^2(U)$.
\end{enumerate}
\end{proof}

Because our aim in this paper is to present a clear construction, not to obtain the sharpest possible results (for example with respect to identifying the precise index sets that arise), we will work with the following spaces of b and $\phi$-pseudodifferential operators in the construction.  In these spaces, we require only certain leading order behaviour of the kernel expansions at the various faces in the blown-up double spaces.  This makes
the construction easier to follow, but of course, it only tells us that the resulting operator has expansions at the various faces satisfying certain bounds on the exponents in the leading terms.

\begin{definition}
\label{def:Psi alpha}
 For $k,\alpha\in\RR$ define the space of b-pseudodifferential operators of order $k$ and weight $\alpha$ as
 $$ \Psi_b^{k,\alpha}(M) = \bigcup\limits_{\calE} \Psi_b^{k,\calE}(M),$$
 where the union is over those index families $\calE$ for $M^2_b$ satisfying
 $\calE_\lf > \alpha,\ \calE_\rf > -\alpha,\ \calE_\bff \geq 0$.
 Similarly, let
 $$ \Psi_\phi^{k,\alpha}(M) = \bigcup\limits_{\calE} \Psi_\phi^{k,\calE}(M),$$
 where the union is over those index families $\calE$ satisfying
 $\calE_\lf > \alpha,\ \calE_\rf > -\alpha,\ \calE_\bff \geq 0,\ \calE_\ff > 0$.
\noindent 
We define the extended version of this space, $\Psi^{k,\alpha}_{\phi,\ext}(M)$, in the analogous way.
\end{definition}
As motivation for this simplification of the full calculi, note that by Theorems \ref{thm:bdd cpct b} and \ref{thm:bdd cpctphi}, operators in $\Psi_{b/\phi}^{k,\alpha}(M)$ are bounded $x^\alpha L^2\to x^\alpha H^m_{b/\phi}$. Also, the adjoint of an operator in $\Psi_{b/\phi}^{k,\alpha}(M)$ is in $\Psi_{b/\phi}^{k,-\alpha}(M)$.

The properties of these spaces that we need in the parametrix construction are collected in the following proposition.  Because these properties are only relevant   near the boundary of $M$ when we construct the
parametrix, in this proposition and for Section 5, 
we use $\Psi_b^{k,\alpha}$ to denote $\Psi_b^{k,\alpha}(V,\calK)$ and $\Psi_\phi^{k,\alpha}$
to denote $ \Psi_\phi^{k,\alpha}(U)$. Also, write 
$\Psi_{b,\ext}^{k,\alpha}:=\Psi_{b,\ext}^{k,\alpha}(U)$, defined in an analogous way as $\Psi_{\phi,\ext}^{k,\alpha}(U)$,  and $\Psi_{\phi,\ext}^{k,\alpha}:=\Psi_{\phi,\ext}^{k,\alpha}(U)$.
\begin{proposition} \label{lem:compos lemma}
 Let $k,l,\alpha,c\in\RR$. 
 Then we have
\begin{enumerate}
 \item[a)]
 $ \Psi_b^{k,\alpha} \circ \Psi_b^{l,\alpha} \subset \Psi_b^{k+l,\alpha}$,
 \item[b)]
  $ \Psi_\phi^{k,\alpha} \circ \Psi_\phi^{l,\alpha} \subset \Psi_\phi^{k+l,\alpha}$,
 \item[c)]
 $\Psi_b^{k,\alpha}\subset \Psi_\phi^{k,\alpha}$ via lifting if $k< 0$,
 \item[d)]
 $x^{-c} \Psi_{b}^{k,\alpha} x^c = \Psi_{b}^{k,\alpha-c}$, 
  $x^{-c} \Psi_{\phi}^{k,\alpha} x^c = \Psi_{\phi}^{k,\alpha-c}$,
 \item[e)]
 $x^\infty \Psi_b^{k,\alpha} = x^\infty \Psi_\phi^{k,\alpha}$, where $x^\infty \Psi_{b/\phi}^{k,\alpha}:=\bigcap\limits_{c\in\RR}x^c \Psi_{b/\phi}^{k,\alpha}$,
 \item[f)]
 $\Psi_b^{k,\alpha}x^c \Psi_\phi^{l,\alpha} \subset \Psi_{b,\ext}^{-\infty,\alpha} + x^c \Psi^{k+l}_{b\phi}$ and\\
  $\Psi_\phi^{k,\alpha}x^c \Psi_\phi^{l,\alpha} \subset \Psi_{b,\ext}^{-\infty,\alpha} + x^c \Psi^{k+l}_{b\phi}$ for $c\geq 0$ and $k\leq 0$, where 
 $$\Psi^{m}_{b\phi}:= \bigcap_{\alpha\in\RR} \Psi^{m,\alpha}_{\phi} $$
 is the space of $\phi$-operators having empty index sets at $\lf,\rf$, index set $\geq0$ at $\bff$ and $>0$ at $\ff$.
 \item[g)] These properties are all still true if we replace $\Psi^{m,\alpha}_{\phi}$ with $\Psi^{m,\alpha}_{\phi, ext}$.
\end{enumerate}
\end{proposition}

Concerning f), note that $\Psi_b^{k,\alpha}x^c \Psi_\phi^{l,\alpha} \subset \Psi^{k+l,\alpha}_{\phi}$ follows from b) and c), but the factor $x^c$ is lost in this crude argument; f) is a refinement needed below.
The main point of the proposition is to record the behavior of kernels and compositions as b- or $\phi$-kernels and their expansions at the left and right faces. We will not make use of the information about their orders.

\begin{proof}
Parts a) and b) follow immediately from the  composition theorems for the b- and $\phi$-calculus. Part c) follows from the lifting theorem. For d) and e) observe that multiplying a b-operator by $x^c$ from the left or right means multiplying its Schwartz kernel by $x^c$ or $(x')^c$, respectively. Therefore, its index sets at $\lf,\bff$, respectively at $\bff,\rf$, are simply shifted up by $c$. This is similar for $\phi$-operators. This implies d), and e) follows from the fact that functions vanishing to infinite order at a front face of a blow-up are still smooth after collapsing that front face.
 
Now we prove f). Let $S\in \Psi_\phi^{l,\alpha}$. By introducing a cutoff function supported in a neighborhood of $\Diag_\phi\cup\ff$ and equal to one in a smaller such neighborhood, where $\Diag_\phi$ is the diagonal in $M^2_\phi$, we can write $S=S_1+S_2$ where $S_1\in\Psi^{-\infty,\alpha}_b$ (since a $\phi$-operator whose kernel vanishes near $\ff$ is a b-operator) and $S_2\in\Psi^l_{b\phi}$. Thus
\begin{equation}
 \label{eqn: phi decomp}
\Psi_\phi^{l,\alpha}\subset \Psi^{-\infty,\alpha}_{b,\ext} + \Psi^l_{b}\,,
\end{equation}
so $\Psi_b^{k,\alpha}x^c \Psi_\phi^{l,\alpha}\subset \Psi_{b,\ext}^{-\infty,\alpha} + \Psi_{b,\ext}^{k,\alpha}x^c\Psi^l_{b\phi}$ by a). 
Next, from $x^c\Psi^l_{b\phi}=\Psi^l_{b\phi}x^c$ and b), c) we get
$\Psi_b^{k,\alpha}x^c \Psi_{b\phi}^{l}\subset \Psi_\phi^{k+l,\alpha}x^c$.
Finally, using \eqref{eqn: phi decomp} again, and $c\geq 0$, we see that this is contained in $\Psi_{b,\ext}^{-\infty,\alpha} + \Psi_{b\phi}^{k+l}x^c$. This gives the first claim in f). For the second claim use \eqref{eqn: phi decomp} to obtain
$\Psi_\phi^{k,\alpha} \subset \Psi_b^{-\infty,\alpha} + \Psi^k_{b\phi}$.  The first term is handled by the first containment in f), and the second yields $\Psi^k_{b\phi}x^c\Psi^{l,\alpha}_\phi = x^c\Psi^k_{b\phi}\Psi^{l,\alpha}_\phi\subset x^c\Psi^{k+l,\alpha}_\phi$ by b). Property g) is proved by essentially the same arguments
as a) through f).
\end{proof}

\section{Statement and proof of main parametrix theorem}
\label{sec:parametrix}
Before we construct parametrices, we should consider for a moment
what sort of parametrices we want.  
Recall that a right parametrix for $P$ is an operator $Q$ so that $PQ=\Id -R$, where the remainder $R$ is `good' in a suitable sense, and similar is true for a left parametrix. Parametrices satisfying more stringent notions of `good' are harder to construct in general than those satisfying weaker notions, but also lead to stronger corollaries. For example, the classical pseudodifferential parametrix construction for an elliptic operator, $P$, over a smooth manifold yields a remainder $R$ which is smoothing.  This may be used to prove that solutions of $Pu=0$ are smooth in the interior of $M$. If $M$ is a compact manifold without boundary, such a remainder is also a compact operator, which shows that $P$ is a Fredholm operator between Sobolev spaces. However, in our setting a smoothing remainder need not be compact on the natural weighted Sobolev spaces on which it is bounded.  Thus in order to get a Fredholm result we need to improve the parametrix. If we would further like to get full asymptotic expansions at the boundary for solutions $u$ of $Pu=0$, this requires an even more refined parametrix. More precisely, in the $\phi$-calculus setting this requires the Schwartz kernel of the left remainder $R$ to be smooth in the interior of $M^2_\phi$ and 
to vanish to infinite order at the faces $\bff$, $\ff$ and also at $\rf$. The expansion of $R$ at $\lf$ then
determines the expansion of $u$.  This is the type of parametrix we will
construct here.
As usual we will first construct a right parametrix, where the remainder has the same properties with $\lf$ and $\rf$ interchanged, and then use adjoints to get a left parametrix.

The regularity and vanishing properties of the remainder kernel correspond to the two parts of the parametrix construction, which, however, we carry out in the opposite order.
\begin{enumerate}
\item[I.]
First we find a right parametrix `at the boundary'. This yields a right parametrix $Q_\partial$ with a remainder $R_\partial$ that vanishes to infinite order at $\bff$, $\ff$
and $\lf$.
\item[II.]
We combine $Q_\partial$ with a `small' right $\phi$-parametrix of $P$, obtained by inverting the principal symbol, to improve the error to be smoothing in addition. This comes at the cost of losing some control of the $\Pi,\Piperp$ splitting, but this is no problem for the Fredholm and regularity results we want.
\end{enumerate}

\subsection{Statement of Main Theorem}

Recall from equation \eqref{eqn:D_U} that the Gauss-Bonnet operator for a $\phi$-cusp metric near the boundary may be written as 
$$
D_U= x^{-a}
\left(
\begin{array}{ll}
x^a P_{00} & x^a P_{01} \\
x^a P_{10} & P_{11}
\end{array}
\right) 
$$
with respect to the decomposition $\Cinfty(V,\calK)\oplus\Cinfty(V,\calC)$ into the space of fibre-harmonic forms and its orthogonal complement. Also, recall from \eqref{eqn:Delta_M} and the discussion after it that the Hodge Laplacian may be written, under the condition that $[D_B,\Pi]=0$, as
$$
\Delta_U =
x^{-2a}\left(
\begin{array}{ll}
x^{2a}T_{00}& x^{2a} T_{01}\\
x^{2a} T_{10}& T_{11}
\end{array}
\right).
$$
If we can obtain a parametrix for $x^aD_U$ and for $x^{2a}\Delta_U$, we will be able to get parametrices for $D_U$ and $\Delta_U$ by composing with appropriate powers of $x$.  Therefore, in this section we will consider the more general situation that 
$P$ is a $\Pi$-split operator, which we define as follows:
\begin{definition}
\label{splitprop}
Let $P$ be a $\phi$-elliptic differential $\phi$-operator of order $m$
\[
P:C^\infty(M,E)\to C^\infty(M,E),
\]
$\calK$ be a finite dimensional sub-bundle of the infinite dimensional $C^\infty(F,E)$ bundle 
over $V$, with orthogonal complement $\calC$, and let $\Pi,\Piperp$
be the projections on the first and second factor in the orthogonal decomposition $\Cinfty(U,E)=\Cinfty(V,C^\infty(F,E))=\Cinfty(V,\calK)\oplus\Cinfty(V,\calC)$.
Write $P$ in terms of this decomposition as
\begin{equation}
\label{eqn:P} 
 P = 
\begin{pmatrix}
 x^{am} P_{00} & x^{am} P_{01} \\
 x^{am} P_{10} & P_{11}
\end{pmatrix},
\end{equation}
where $x^{am}P_{00} = \Pi P \Pi$ etc.
We say that $P$ is {\em $\Pi$-split} if the following conditions are satisfied.
\begin{enumerate}
\item $P_{00}\in \Diff^m_b(V,\calK)$ is an elliptic b-operator.
\item There is $\Ptilde\in \Diff_\phi^m(U,E)$ such that $P_{01}=\Pi \Ptilde \Piperp$ and $P_{10} = \Piperp\Ptilde\Pi$.
\item $N(P_{11})$ is invertible on forms over $\mathbb{R}^{b+1}$ with values
in $\calC$.
\end{enumerate}
\end{definition}
\noindent
Note that condition 2 implies that $N(P)$ is diagonal with respect to the decomposition \(\Cinfty(\RR^{b+1}\times F,E)=\Cinfty(\RR^{b+1},\calK)\oplus\Cinfty(\RR^{b+1},\calC)\), so condition 3 makes sense. Also, the conditions imply that $P_{01}, P_{10}, P_{11}\in \Psi^m_{\phi,\ext} (U)$.

As discussed in Sections \ref{sec:laplacian} and \ref{sec:ext phi calculus}, these conditions are satisfied by $x^aD_M$ and $x^{2a}\Delta_M$.

As in the b-calculus, we will not generally be able to construct a parametrix for such operators that gives Fredholm results for all weighted $L^2$ spaces.
Rather, we will find a family of parametrices corresponding to a dense set of weights.  In fact, these weights come from the same source as the weights in the b setting, as will become clear in the 
construction.  For any
admissible weight $\alpha$, the corresponding parametrix will allow us to prove that 
$$ P: x^\alpha H^m_\split \to x^\alpha L^2 $$
is Fredholm. The Sobolev space $H^m_\split$ that appears on the left side of this map 
is introduced in Section \ref{subsec:split sob}.

We will prove the following Theorem. We use the short notation
$\Psi_b^{k,\alpha}$, $\Psi_{\phi,\ext}^{k,\alpha}$ and $\Psi_{b\phi,\ext}^{k}$  for  $\Psi_b^{k,\alpha}(V,\calK)$,
$\Psi_{\phi,\ext}^{k,\alpha}(U,E)$ and $\Psi_{b\phi,\ext}^{k}(U,E)$, respectively. These spaces were defined in Definition \ref{def:Psi alpha} and Proposition \ref{lem:compos lemma}. 
\begin{theorem}
 \label{thm:right split param}
 Let $P$ be a $\Pi$-split differential operator of order $m$ as in Definition \ref{splitprop}. Let $\alpha\in\RR$ and assume
\begin{equation}
 \label{eqn:weight condition}
  \alpha-am\not\in -\imspec (P_{00})\,.
\end{equation}
Then there are right and left parametrices which in the interior of $M$ are pseudodifferential operators of order $-m$ and over $U$ are in the spaces
 $$ Q_{r,\alpha}, Q_{l,\alpha} \in
\begin{pmatrix}
 x^{-am} \Psi_b^{-m,\alpha} + \Psi_{b\phi,\ext}^{-m}& 0 \\
 0 & \Psi_{\phi,\ext}^{-m} + \Psi_{\phi,\ext}^{-m,\alpha}x^{am}
\end{pmatrix} +
\begin{pmatrix}
 0 & x^{-am}\Psi^{-m,\alpha}_{\phi,\ext}  x^{am} 
 \\
\Psi^{-m,\alpha}_{\phi,\ext} 
 & 0
\end{pmatrix} $$
with 
$$PQ_{r,\alpha} = \Id - R_{r,\alpha},\quad Q_{l,\alpha}P= \Id - R_{l,\alpha},$$
where the remainders are smooth in the interior and over $U$ satisfy
\begin{align}
 R_{r,\alpha} &\in x^\infty \Psi_\phi^{-\infty,\alpha}(\Pi + x^{am}\Piperp), \label{eqn:Rright} \\
 R_{l,\alpha} &\in  (x^{-am}\Pi + \Piperp)\Psi_\phi^{-\infty,\alpha}x^\infty.\label{eqn:Rleft}
\end{align}
\end{theorem}
The notation in \eqref{eqn:Rleft} means that $R_{l,\alpha}$ is a sum of the form $x^{-am}\Pi R' x^\infty + \Piperp R'' x^\infty$
with $R',R'' \in\Psi_\phi^{-\infty,\alpha}$, and the notation in \eqref{eqn:Rright} is the same.
It will be clear from the construction that the index sets for the various parts of the parametrix could be described more precisely if required and will be computed from $\Specb (P_{00})$.

\subsection{Construction of right parametrix}
In order to make this section easier to read, we leave out the subscript $\ext$ in the notation for spaces of pseudodifferential operators. Thus, $\Psi^*_\phi$ is to be understood as $\Psi^*_{\phi,\ext}$ throughout this section.

Assume that $P$ is $\Pi$-split.  Decompose $P$ into its diagonal and off-diagonal parts: 
$$P=P_d + P_o,\quad P_d = \begin{pmatrix}
 x^{am} P_{00}  & 0 \\
 0 & P_{11}
\end{pmatrix},\quad
P_o = 
\begin{pmatrix}
 0 & x^{am} P_{01} \\
 x^{am} P_{10} & 0
 \end{pmatrix}.
 $$
The construction now proceeds in five steps.  The first four steps are needed to construct
the boundary parametrix, and in the fifth step, we combine this with the interior parametrix
to make the remainder smoothing.
\begin{itemize}
 \item[Step 1:]
First, we find a parametrix, $Q_d$, for the diagonal part, $P_d$, of $P$. Since $P_{00}$ is a b-elliptic b-operator, it has a parametrix in the b-calculus, which we set as the upper left piece in $Q_d$. For the lower right part we use the invertibility of $N(P_{11})$ and Proposition \ref{lem: extended phi-calc}.3. The diagonal remainder, $R_d:= P_dQ_d-I$ vanishes to infinite order in $x$ at all boundary faces, but
$Q_d$ is not a parametrix for the full operator $P$.
 \item[Step 2:]
 Using the off-diagonal terms of $P$ we modify $Q_d$ to $Q_2$ so that $PQ_2$ yields a remainder $R_2$ that vanishes to some order at $\bff$ and $\ff$.
 \item[Step 3:]
By constructing formal solutions of the equation $Pu=f$, we construct a right parametrix $Q_3$ with a remainder $R_3$ that vanishes to infinite order at $\lf$, but still has non-trivial index sets at $\bff$ and $\ff$.
However, the remainder does vanish to some order at these faces, so we can correct this
in the next step.
\item[Step 4:]
We use a Neumann series to get a right parametrix $Q_\partial$ with a remainder $R_\partial$ that vanishes to infinite order at $\bff$, $\ff$ and $\lf$.
This finishes the construction of the boundary parametrix.
\item[Step 5:]
By combining $Q_\partial$ with a `small' right $\phi$-parametrix of $P$, we get a remainder that in addition is smoothing.
\end{itemize}
We now provide the details.  
\begin{itemize}
 \item[Step 1:]
Here we use the assumption that the normal operator of $P$ is diagonal. We only use the $\Piperp$ part of the normal operator, since in its $\Pi$ part the normal operator forgets  important information. For example, if the $\Pi$ part is of the form $x^{1+a}\partial_x + x^ah$ (with $h$ a zero order operator) then the normal operator will be $x^{1+a}\partial_x$, while when we  write this as $x^a(x\partial_x + h)$,  then we see that $h$ is included in the indicial operator of $P_{00}=x\partial_x + h$, so it is essential for the correct b-parametrix.
Therefore, in the $\Pi$-part we use the b-parametrix. 

In order to construct the right parametrix, it is useful to write $\Pi P\Pi$ as $P_{00}'x^{am}$ instead of $x^{am}P_{00}$.
Thus, define the b-operator  
$$P'_{00}=x^{am}P_{00}x^{-am}$$
and choose a right parametrix $Q_{00}$ for $P'_{00}$ in the b-calculus, corresponding to the given weight $\alpha$, see Theorem \ref{thm: b-parametrix}. This is possible under the condition
$ \alpha \not\in  -\imspec P_{00}' $. Since $-\imspec x^{am}P_{00}x^{-am} = -\imspec P_{00} + am$ this condition is equivalent to condition \eqref{eqn:weight condition}. We obtain
$$ x^{am}P_{00} x^{-am}Q_{00} = \Id - R_{00},
\quad Q_{00}\in \Psi^{-m,\alpha}_b,\
R_{00}\in x^\infty\Psi^{0,\alpha}_b.
$$
We now invert $N(P_{11})$ on $\calC$ and define its inverse $B$ as in \eqref{eqn:N(P) inverse}. Then by Proposition \ref{lem: extended phi-calc}.3, $B\in\Psi^{-m}_{\sus-\phi}(\partial M)$, and by 
Proposition \ref{lem: extended phi-calc}.4 we may extend $B$ to the interior of $M^2_\phi$  and obtain a parametrix $Q_{11}'$:
$$ P_{11} Q_{11}' = \Id - R_{11}',\quad
Q_{11}'\in \Psi^{-m}_{\phi},\
R_{11}'\in x\Psi^0_{\phi}.$$
Using Proposition \ref{lem: extended phi-calc}.5 and the standard Neumann series argument this can be improved to 
$$ P_{11} Q_{11} = \Id - R_{11},\quad
Q_{11}\in \Psi^{-m}_{\phi},\
R_{11}\in x^\infty\Psi^0_{\phi}.$$

Now we have $P_dQ_d = \Id - R_d$ with
\begin{equation}
\label{eqn:Qd,Rd1} 
 Q_d = 
\begin{pmatrix}
 x^{-am}Q_{00} & 0 \\
 0 & Q_{11}
\end{pmatrix}
\in
\begin{pmatrix}
 x^{-am} \Psi_b^{-m,\alpha} & 0 \\
 0 & \Psi_{\phi}^{-m}
\end{pmatrix}
,
\end{equation}
\begin{equation}
\label{eqn:Qd,Rd} 
R_d = 
\begin{pmatrix}
 R_{00} & 0 \\
 0 & R_{11}
\end{pmatrix}
\in
\begin{pmatrix}
x^\infty
 \Psi^{0,\alpha}_\phi & 0 \\
 0 & x^\infty
\Psi^0_{\phi}
\end{pmatrix}.
\end{equation}

 \item[Step 2:]
In this step we make use of the fact that the $x^{am}$ factors in $P_o$  mean than it is of lower order near the boundary than $P_d$. The idea comes from the formal inversion 
\[P^{-1} = P_d^{-1}(\Id + P_oP_d^{-1})^{-1} = P_d^{-1} - P_d^{-1}P_oP_d^{-1} + \dots.
\] 
Since $P_d$ may not be invertible we use its parametrix $Q_d$ instead. The first two terms suffice for our purposes. Thus, we set
$$ Q_2 = Q_d + Q_o,\quad \mbox{where} \quad Q_o := -Q_dP_o Q_d. $$
Then $P_d Q_d = \Id - R_d$ implies 
$$PQ_2 = 
\Id - R_2,\quad \mbox{where } R_2=R_d - R_o + (P_o Q_d)^2,\quad \mbox{and }R_o:=R_dP_o Q_d.$$
To analyze $Q_2$ and $R_2$ we use Proposition \ref{lem:compos lemma} to get
$ \Psi^m_{\phi}\Psi^{-m,\alpha}_b \subset \Psi^{0,\alpha}_{\phi}$.  Together with $x^{am}\Psi_{\phi}^0=\Psi_{\phi}^0 x^{am}$ we get 
\begin{equation}
\label{eqn:PoQd} 
 P_oQ_d \in 
\begin{pmatrix}
 0 & \Psi_{\phi}^0 x^{am} \\
\Psi^{0,\alpha}_{\phi} & 0
\end{pmatrix}.
\end{equation}
This gives
\begin{equation}
\label{eqn:Q_o} 
 Q_o = Q_d P_o Q_d \in 
\begin{pmatrix}
 0 & x^{-am}\Psi^{-m,\alpha}_{\phi}  x^{am} 
 \\
\Psi^{-m,\alpha}_{\phi}
 & 0
\end{pmatrix}.
\end{equation}
Now we analyze $R_2$.
From \eqref{eqn:Qd,Rd} and \eqref{eqn:PoQd} we get
\begin{equation}
\label{eqn:R2 terms} 
R_o= R_d P_oQ_d \in 
\begin{pmatrix}
 0 &  x^\infty
\Psi^{0,\alpha}_\phi x^{am} \\
x^\infty
\Psi^{0,\alpha}_{\phi} 
 & 0
\end{pmatrix},\quad
(P_oQ_d)^2 \in  
\begin{pmatrix}
x^{am}\Psi^{0,\alpha}_{\phi} 
 & 0 \\ 0 & \Psi^{0,\alpha}_{\phi} x^{am}
\end{pmatrix}.
\end{equation}
The overall $x^{am}$ factor in $(P_oQ_d)^2$ (as opposed to $x^{am}$ in one entry only in \eqref{eqn:PoQd}) is the reason that $Q_2$ is better than $Q_d$. Note that this factor increases the order of vanishing of the Schwartz kernel at $\bff$ and $\ff$, no matter whether it is on the left or on the right. However, on the left it increases the order at $\lf$ but not at $\rf$, while on the right it increases the order at $\rf$ but not at $\lf$.  Thus in the second expression in equation \eqref{eqn:R2 terms}, we cannot move both factors of $x^{am}$ to the same side of the pseudodifferential term.

 \item[Step 3:]
 The parts $R_d$ and $R_dP_oQ_d$ in $R_2$ already have the infinite order vanishing at $\lf,\bff,\ff$ that we need. However, the term $(P_oQ_d)^2$ does not, see \eqref{eqn:R2 terms}. The standard procedure to improve remainder terms is a Neumann series argument as in Step 4, which replaces a remainder $R$ by arbitrarily high powers $R^N$. However,  as can be seen from the Composition Theorem \ref{thm:composition phi}, taking powers of $R$ does not improve order of vanishing at $\lf$, and also does not improve order of vanishing at $\bff,\ff$ unless $R$ already vanishes at $\lf$ to arbitrarily high order. 
There is a standard remedy for these problems. By constructing formal solutions $u$ of the equation $Pu=f$, where $f$ arises from the $\lf$ expansion of the remainder, combining these solutions into a Schwartz kernel and subtracting the resulting operator from the parametrix, one gets an improved remainder, which vanishes to infinite order at $\lf$. Taking powers of this remainder then yields arbitrarily high orders of vanishing at $\bff$ and $\ff$ also. 
This standard procedure is explained in \cite[Section 5.20]{Me-aps}. The setting there is for a b-elliptic b-operator $P$, and we need to adapt this to our situation. We obtain the following proposition. Its proof will appear in \cite{GH3}.
\begin{proposition}
 \label{lem:solve at lf}
 Let $P$ be $\Pi$-split. Suppose the Schwartz kernel of $R'\in\Psi_b^{-\infty,\alpha}$ is supported near $\lf$ and satisfies, with index sets listed in the order $\lf,\bff$,  
$$
\begin{array}{llll}
 \Pi R' & \text{ has index sets } & >\alpha + am, & \geq am, \\
\Piperp R' & \text{ has index sets } &  >\alpha , & \geq am.
\end{array}
$$ 
Then there are operators $Q'\in\Psi_b^{-\infty,\alpha}$, $R''\in\Psi_b^{-\infty,\infty}$ with Schwartz kernels supported near $\lf$ and satisfying
$$
\begin{array}{llll}
 \Pi Q' & \text{ has index sets } & >\alpha, & \geq 0, \\
\Piperp Q' & \text{ has index sets } &  >\alpha , & \geq am, \\
 R'' & \text{ has index sets } & \emptyset, &\geq am,
\end{array}
$$
$$ PQ'=R'+R''.$$
\end{proposition}
We apply the proposition with $R'=\chi R_2$ (understood in the sense of kernels), where $\chi$ is a smooth cut-off function on the blown up double space supported and equal to one near $\lf$. Then \eqref{eqn:Qd,Rd}, \eqref{eqn:R2 terms} imply that $R'$ satisfies the conditions in Proposition \ref{lem:solve at lf}. 
With $Q',R''$ from the proposition we set
$$ Q_3 = Q_2 + Q',\quad R_3 = (1-\chi)R_2 - R'' .$$
Then $PQ_3 = \Id - R_3$ since $PQ_2 = \Id - R_2$, $PQ' = \chi R_2 + R''$.
Note from \eqref{eqn:Qd,Rd}, \eqref{eqn:R2 terms} and Proposition \ref{lem:solve at lf} that $R_3$ is a $\phi$-operator of weight $\alpha$ having index sets $\emptyset$ at $\lf$, $\geq am$ at $\bff$ and $> am$ at $\ff$, and so that $R_3\Piperp$ has an extra $am$ degrees of vanishing at $\rf$.  Equivalently
\begin{equation}
\label{eqn:R3 space} 
 R_3 \in \Psi_R :=
\begin{pmatrix}
 x^{am}\Psi^{0,\alpha}_{\phi,\lf}
& \Psi^{0,\alpha}_{\phi,\lf} x^{am}
\\
\Psi^{0,\alpha}_{\phi,\lf}
&
\Psi^{0,\alpha}_{\phi,\lf} x^{am}
\end{pmatrix},
\end{equation}
where $\Psi^{0,\alpha}_{\phi,\lf}$ is the space of those elements of 
$\Psi^{0,\alpha}_{\phi}$ that vanish to infinite order at $\lf$. 

\item[Step 4:]
In order to improve the remainder further, we use the standard Neumann series argument, i.e. we multiply $PQ_3=\Id-R_3$ by $\Id + R_3 + R_3^2 + \dots$ from the right and sum asymptotically. 
We need to check that the asymptotic sum makes sense.  Clearly $\Psi^{0,\alpha}_{\phi,\lf}x^{am}\subset x^{am}
\Psi^{0,\alpha}_{\phi,\lf}$, and this implies 
that $R_3^2 \in \begin{pmatrix}
 x^{am}\Psi^{0,\alpha}_{\phi,\lf}
& x^{am}\Psi^{0,\alpha}_{\phi,\lf} x^{am}
\\
x^{am}\Psi^{0,\alpha}_{\phi,\lf}
&
\Psi^{0,\alpha}_{\phi,\lf} x^{am}
\end{pmatrix}
$
and then inductively that $R^{2N}_3 \in x^{(N-1)am} \Psi_R$ for all $N\in\NN$.
Therefore, the index sets at $\bff$, $\ff$ of $R^N_3$ wander off to infinity as $N\to\infty$. 
In addition, the index sets of $R^N_3$ at $\rf$ stabilize by a simple argument as in \cite[Section 5.22]{Me-aps}, hence the asymptotic sum $R_3'=\sum_{N=1}^\infty R_3^N$ makes sense and yields an operator in $\Psi_R$, see \eqref{eqn:R3 space}. Setting $Q_\partial = Q_3R_3'$ we obtain
$PQ_\partial = \Id - R_\partial$ where $R_\partial$ is in the intersection of all the spaces $x^{(N-1)am} \Psi_R$, hence 
\begin{equation}
 \label{eqn:R partial}
R_\partial\in x^\infty \Psi^{0,\alpha}_\phi
(\Pi + x^{am}\Piperp).
\end{equation}

We now analyze $Q_\partial=Q_3R_3'$. Recall 
$Q_3=Q_d+Q_o+Q'$ and write $R_3' = \Rtilde_d + \Rtilde_o$, where
$\Rtilde_d$ is diagonal and 
$\Rtilde_o$ is off-diagonal with respect to the $\Pi,\Piperp$-splitting. We analyze each term in the product $Q_3R_3'$ separately, using \eqref{eqn:Qd,Rd}, \eqref{eqn:Q_o} and \eqref{eqn:R3 space}. The $\Pi\,\Pi$ parts of $Q_d\Rtilde_d$, $Q_o\Rtilde_o$ are in $x^{-am}\Psi^{-m,\alpha}_b x^{am}\Psi^{0,\alpha}_{\phi,\lf}$ and $x^{-am}\Psi_\phi^{-m,\alpha}x^{am}\Psi^{0,\alpha}_{\phi,\lf}$ respectively, hence by Proposition \ref{lem:compos lemma}f) are contained in
$x^{-am}\Psi^{-\infty,\alpha}_b + \Psi^{-m}_{b\phi}$. 
The $\Piperp\Piperp$ terms are unproblematic, and we get
\begin{equation}
\label{eqn:Q3R3 diag}
 Q_d \Rtilde_d + Q_o\Rtilde_o \in 
\begin{pmatrix}
 x^{-am}\Psi^{-\infty,\alpha}_b + \Psi^{-m}_{b\phi} & 0 \\
 0 & \Psi^{-m,\alpha}_{\phi} x^{am}
\end{pmatrix} .
\end{equation}
Next, by a similar argument, we get
\begin{equation}
\label{eqn:Q3R3 offdiag} 
 Q_d\Rtilde_o + Q_o \Rtilde_d \in 
\begin{pmatrix}
 0 & x^{-am} \Psi_\phi^{-m,\alpha} x^{am} \\
 \Psi_\phi^{-m,\alpha} & 0
\end{pmatrix}.
\end{equation}
Finally, Proposition \ref{lem:solve at lf} gives in particular
\begin{equation}
 \label{eqn:Q'}
 Q' \in 
\begin{pmatrix}
 \Psi_{b,\rf}^{-\infty,\alpha} & \Psi_{b,\rf}^{-\infty,\alpha}\\
 \Psi_{b,\rf}^{-\infty,\alpha} x^{am} & \Psi_{b,\rf}^{-\infty,\alpha} x^{am}
\end{pmatrix},
\end{equation}
where $\Psi_{\phi,\rf}^{-\infty,\alpha}$ are those operators in $\Psi_{\phi}^{-\infty,\alpha}$ whose kernels vanish at $\rf$ to infinite order. Then with
$R_3'\in \Psi_R$ from Equation \ref{eqn:R3 space} we get
\begin{equation}
\label{eqn:Q'R3} 
 Q'R_3' \in 
\begin{pmatrix}
\Psi_{\phi}^{-\infty,\alpha} & \Psi_{\phi}^{-\infty,\alpha} x^{am} \\
\Psi_{\phi}^{-\infty,\alpha} & \Psi_{\phi}^{-\infty,\alpha} x^{am} 
\end{pmatrix}
=\Psi^{-\infty,\alpha}_\phi (\Pi + x^{am}\Piperp).
\end{equation}

\item[Step 5:]
Now we improve the parametrix further using a small $\phi$-parametrix of $P$. That is, since $P$ is $\phi$-elliptic,  we may find a parametrix by  Theorem \ref{thm: phi-parametrix}
$$ P Q^\sigma = \Id -R^\sigma,\quad Q^\sigma \in \Psi^{-m}_\phi,
\quad R^\sigma \in \Psi^{-\infty}_\phi. $$

Then we define
 $$Q_r=Q_\partial + Q^\sigma R_\partial,\quad R_r= R^\sigma R_\partial. $$
 so $PQ_r = \Id-R_r$. 
\end{itemize}

\subsection{Construction of left parametrix and proof of Theorem \ref{thm:right split param}}
The rough idea for constructing a left parametrix is as follows. Let $P^*$ be the formal adjoint of $P$. Then
$P^*$ is also $\Pi$-split, so we may construct a right parametrix $Q_r'$ as above but for $P^*$, and obtain
a remainder $R_r'$, so $P^*Q_r' = \Id - R_r'$. Taking adjoints we obtain
$Q_l P = \Id - R_l$ where $Q_l = (Q_r')^*$, $R_l = (R_r')^*$.

To show that the formal adjoint of a $\Pi$-split operator is again $\Pi$-split, consider what taking the formal adjoint of $P$ means. It means that the Schwartz kernel of $P$ is reflected across the diagonal, so coordinates are switched $x\leftrightarrow x', y\leftrightarrow y', z\leftrightarrow z'$. Since $\Pi$ is an orthogonal projection, the $2\times 2$ matrices representing the $\Pi,\Piperp$ decomposition are simply flipped, and adjoints of its parts are taken.

The $\Pi\,\Pi$  part of $P^*$ is  $(P_{00})^* x^{am}$. Now $\imspec (P_{00})^* = -\imspec(P_{00})$, so by assumption \eqref{eqn:weight condition} we have $am-\alpha\not\in-\imspec(P_{00})^*$. Therefore, there is a right b-parametrix 
for $(P_{00})^*$ for the weight $am-\alpha$, and we can construct a right parametrix $Q_r'$ as above for $P^*$ and for this weight. Let $Q_l = (Q_r')^*$.

\begin{proof}[Proof of Theorem \ref{thm:right split param}]
 We collect all terms in $Q_r=Q_\partial + Q^\sigma R_\partial $. First, $Q_\partial$ is the sum
 of $Q_d$ in \eqref{eqn:Qd,Rd}, $Q_o$ in \eqref{eqn:Q_o} and the terms in \eqref{eqn:Q3R3 diag}, \eqref{eqn:Q3R3 offdiag}, \eqref{eqn:Q'} and \eqref{eqn:Q'R3}, and each one is in the space given for $Q_D+Q_O$. Here the $\Pi\,\Pi$ term in $Q'R_3$ is split up as in \eqref{eqn: phi decomp}. Then from \eqref{eqn:R partial} we have
 $Q^\sigma R_\partial \in R_\partial\in x^\infty \Psi^{-m,\alpha}_\phi
(\Pi + x^{am}\Piperp)$, which is also of the form given in the theorem. 
In the same way we get the claim for $R_r=R^\sigma R_\partial$.
Looking at the left parametrix constructed above we see
$$(x^{-am}\Psi_b^{-m,am-\alpha})^*=\Psi_b^{-m,\alpha-am}x^{-am} = x^{-am} \Psi_b^{-m,\alpha},
$$
and similar calculations show that the left parametrix we obtain is precisely of the same type as the right parametrix. 
The left remainder is in
$$ R_l \in \left(x^\infty\Psi_\phi^{-\infty,am-\alpha} 
(\Pi + x^{am}\Piperp)\right)^*
=
(\Pi + x^{am}\Piperp)\Psi_\phi^{-\infty,\alpha-am}
x^\infty
=
(x^{-am}\Pi + \Piperp)
\Psi_\phi^{-\infty,\alpha}
x^\infty.
$$ 
\end{proof}

\section{Proofs of Fredholm and regularity results}
\label{sec:proofs main thms}
In this section, we first define the Sobolev spaces which reflect the different regularity of fibre harmonic and fibre perpendicular forms, and then prove Fredholm and regularity results.

\subsection{Split Sobolev spaces} \label{subsec:split sob}
When dealing with the Gauss-Bonnet and Hodge Laplace operators, we naturally encounter Sobolev spaces which encode different sorts of regularity in the fibre-harmonic and fibre-perpendicular parts of sections near
the boundary. 
Recall the definition of the bundle $\calK$ in Section \ref{sec:laplacian}.
\begin{definition}
  Let $M$ be a $\phi$-manifold. For any $k\in\RR$ define
  $$\Pi H^k_b(M,E,\dvol_b)= \{\mu\in H_{loc}^k(\Mint,E,\dvol_b) :\, \mu_{|U}\in H^k_b(V,\calK,\dvol_b)\}$$
  and let
  $ \Pi^\perp H^k_\phi (U,E,\dvol_b)$ denote the image of the map $\Pi^\perp$.
Then define the {\em split Sobolev spaces} for any $k\in\RR$ by
$$ H^k_\split (M,E,\dvol_b) = x^{-ak}\Pi H^k_b(M,E,\dvol_b) + \chi\,\Pi^\perp H^k_\phi (U,E,\dvol_b),$$
where $\chi\in C_0^\infty(U)$ equals one near $\partial M$.
\end{definition}
In the sequel we will often write $H^k_\phi$, $H^k_\split$ for $H^k_\phi (M,E,\dvol_b)$, $H^k_\split (M,E,\dvol_b)$ etc.
That these are the natural Sobolev spaces for $\Pi$-split operators can be seen in the special case where the operator $P$ in \eqref{eqn:P} is diagonal, i.e. $P_{01}=P_{10}=0$. Then $P_{00}$ maps between b-Sobolev spaces and $P_{11}$ between $\phi$-Sobolev spaces.
\noindent
We can notice three important things about these spaces.
\begin{enumerate}
\item These spaces are complete inner product spaces under an inner product that depends
on the metric on $M$ and $E$, on the volume form $\dvol_b$ and on the cut-off function $\chi$.  The topology induced by this inner product is as usual independent of these choices. 
\item $H^0_\split = L^2$ and there are continuous inclusions
\begin{equation}
\label{eqn:split containments}
\begin{aligned}
 H^k_\phi  \subset H^k_\split \subset x^{-ak}H^k_b & \text{ for } k \geq 0.
\end{aligned}
 \end{equation}
This follows from $H^k_\phi\subset x^{-ak}H^k_b$, which is obvious from the definitions, by applying it to the $\Pi$ and $\Piperp$ parts separately.
\item Unlike usual Sobolev spaces, the split Sobolev spaces do not form a scale of spaces, i.e.
\begin{equation}\label{eqn:Hsplit not a scale}
 H^{k+1}_\split  \not\subset H^k_\split 
\end{equation}
in general (that is, if $\calK\neq 0$) since higher regularity comes with potentially greater rate of blow-up of the 
fibre harmonic part of a section. This is the reason that our Fredholm theorem below holds only as a map $H^m_\split\to L^2$ and not, as might be expected, $H^{m+k}_\split\to H^k_\split$ for all $k\in\RR$.
\end{enumerate}

\subsection{Proof of main theorems}
\label{subsec:proof main}
We first prove  Fredholm and  regularity theorems for general $\Pi$-split operators $P$. Then we apply the results to deduce Theorems \ref{Thm1} and \ref{Thm2}.
\begin{theorem}
\label{thm:fredholm}
 Let $P$ be a $\Pi$-split operator of order $m$. Let $\alpha\in\RR$. If $\alpha-am\not\in -\imspec(P_{00})$ then $P$ is Fredholm as an operator $x^\alpha H^m_\split(M,E,\dvol_b) \to x^\alpha L^2(M,E,\dvol_b)$.
\\
Also, if $\alpha\not\in -\imspec(P_{00})$ then $P$ is Fredholm $x^\alpha L^2(M,E,\dvol_b)
\to x^\alpha H^{-m}_\split(M,E,\dvol_b)$.
\end{theorem}
Note that the conditions on $\alpha$ are the natural ones when only considering the $\Pi\,\Pi$ part of the operator. In the first statement, $x^{am}P_{00}$ is to map $x^{\alpha-am}H^m_b\to x^\alpha L^2$, so $P_{00}$ is to map $x^{\alpha-am}H^m_b\to x^{\alpha-am} L^2$, which is Fredholm if $\alpha-am\not\in-\imspec(P_{00})$. In the second statement, $x^{am}P_{00}$ is to map $x^\alpha L^2\to x^{\alpha+am}H^{-m}_b$, 
so $P_{00}$ is to map $x^\alpha L^2\to x^{\alpha}H^{-m}_b$, which is Fredholm if $\alpha\not\in-\imspec(P_{00})$.  
\begin{proof}
Conjugating by a power of $x$, we may assume $\alpha=0$. We leave out the index $\alpha$ for parametrices and remainders.
First, we show that $P$ is bounded $H^m_\split \to L^2$. We look at each term in \eqref{eqn:P} separately. Clearly, $x^{am}P_{00}=P_{00}'x^{am}:x^{-am}\Pi H_b^m\to L^2$ and $P_{11}: \Piperp H^m_\phi\to L^2$, so it remains to check the cross terms. Boundedness $x^{am}P_{01}: H^m_\phi \to L^2$ follows from $P_{01}\in\Psi^m_{\phi,\ext}$, and boundedness $x^{am}P_{10}:x^{-am} H^m_b\to L^2$ follows from  $P_{10} =\Piperp \Ptilde\Pi$ with $\Ptilde\in\Diff^m_\phi\subset\Diff^m_b$. 

Next, we check boundedness of the left and right parametrix $Q_l,Q_r:L^2\to H^m_\split$. First, consider the $\Pi\,\Pi$ terms. Clearly, any element of $x^{-am}\Psi_b^{-m,0}$ is bounded $L^2\to x^{-am}H_b^m$. Also, an element of $\Psi^{-m}_{b\phi}$ is bounded $L^2\to H_\phi^m$, so also into $x^{-am}H_b^m$. Next, the $\Piperp\Piperp$ terms are in $\Psi_\phi^{-m,0}$, hence bounded $L^2\to H^m_\phi$.
The $\Pi\,\Piperp$ term is in $x^{-am}\Psi_\phi^{-m,0}x^{am}$, and we need to show that it is bounded $L^2\to x^{-am}H^m_b$. Split the operator as in \eqref{eqn: phi decomp}. Then the first part is in $x^{-am}\Psi_{b,\ext}^{-\infty,0}$, hence ok, and the second part is in $\Psi_\phi^{-m,0}$, hence maps $L^2\to H^m_\phi$ and therefore into $x^{-am}H^m_b$.
Finally, the $\Piperp\Pi$ part of the parametrix is in $\Psi_\phi^{-m,0}$, hence bounded $L^2\to H^m_\phi$ as required.

It remains to prove compactness of the remainders. First, compactness of  $R_r:L^2\to L^2$ follows from Theorem \ref{thm:bdd cpctphi}.2.
To prove the compactness of $R_l:H^m_\split \to H^m_\split$, we first map
$$H^m_\split \hookrightarrow x^{-am}H^m_b \stackrel{x^{am}}{\longrightarrow} H^m_b\,.$$ 
We now consider $\Pi R_l$ and $\Piperp R_l$ separately. Using \eqref{eqn: phi decomp} again, we split 
$\Pi R_l$ as a sum of $R_{l}'\in x^{-am}\Psi_b^{-\infty,0}x^\infty$ and  $R_{l}''\in x^\infty \Psi_{b\phi}^{-\infty,0}$. Since elements of $\Psi_b^{-\infty,0}x^\infty$ are compact as operators $H^m_b\to H^m_b$ by Theorem \ref{thm:bdd cpct b}, the operator $R_{l}':H^m_b\to x^{-am}H_b^m$ is compact.  For the remaining terms $R_{l}''$ and $\Piperp R_l$, we first map $H^m_b\hookrightarrow H^m_\phi$. Both of these operators are compact $H^m_\phi\to H^m_\phi$ by Theorem \ref{thm:bdd cpctphi}, hence the inclusion $H^m_\phi\to H^m_\split$ completes the proof.

Finally, to prove the last statement, apply the first statement to $P^*$ and the weight $-\alpha$.  This is possible under the condition 
$-\alpha\not\in-\imspec((P_{00})^*)=\imspec(P_{00})$. Then take adjoints and identify the dual spaces of $x^{-\alpha}L^2$, $x^{-\alpha}H^m_\split$ with $x^{\alpha}L^2$, $x^{\alpha} H^{-m}_\split$ via the $L^2$ scalar product.
\end{proof}

\begin{theorem}
\label{thm:regularity}
 Suppose $u\in x^\alpha H^m_\split(M,E,\dvol_b)$ for some $\alpha\in\RR$ and $Pu=0$, where the operator $P$ is $\Pi$-split of order $m$. Then $\Pi u \in x^{-am} \calA^K_\phg(U,E)$ and $\Piperp u \in \calA^K_\phg(U,E)$ for some index set $K>\alpha$ determined by $\Specb(P_{00})$.

Also, if $u\in x^\alpha L^2(M,E,\dvol_b)$ and $Pu=0$ then $\Pi u\in \calA^K_\phg(U,E) $, $\Piperp u\in x^{am}\calA^K_\phg(U,E)$
with $K>\alpha$ determined by $\Specb(P_{00})$.
\end{theorem}
In addition, $u$ is smooth in the interior of $M$ by elliptic regularity.
\begin{proof}
First, assume $\alpha-am\not\in-\imspec(P_{00})$. The identity $Q_{l,\alpha}P = \Id - R_{l,\alpha}$ holds on $x^\alpha H^m_\split$, since it holds on $C_0^\infty$ and both sides are continuous in $x^\alpha H^m_\split$ by the proof of Theorem \ref{thm:fredholm}. Hence it may be applied to $u$. From $Pu=0$ it follows that $u=R_{l,\alpha}u$. 
Now $\Psi^{-\infty,\alpha}_\phi x^\infty=\Psi^{-\infty,\alpha}_b x^\infty$ as in Proposition \ref{lem:compos lemma}e), and the result follows from \eqref{eqn:Rleft} and the mapping properties of b-operators. If $\alpha-am\in-\imspec(P_{00})$ then apply this argument with $\alpha$ replaced by $\alpha-\eps$ where $\eps>0$ is such that $\alpha-\eps\not\in-\imspec(P_{00})$ and the resulting index set $K$ has no powers in the interval $ (\alpha-\eps,\alpha)$. Then the expansion of $u$ has powers with real parts $\geq \alpha$, but those terms with equality must vanish since they are not in 
$x^\alpha H^m_\split$. 

For the second statement, observe that 
$Q_{l,\alpha+am} $ is bounded $x^\alpha H^{-m}_\split \to x^\alpha L^2$. This follows from the same arguments as in the proof of Theorem \ref{thm:fredholm}. Therefore, the identity $Q_{l,\alpha+am}P = \Id - R_{l,\alpha+am}$ holds on $x^\alpha L^2$. Therefore $u=R_{l,\alpha+am}u$, and the claim follows as before. 
\end{proof}

Before we prove Theorems \ref{Thm1} and \ref{Thm2} we need to define the operator $D_V$ used there.
The operator $P_{00}$ 
in \eqref{eqn:P_00} acts on sections of $\calK$, defined after \eqref{eqn:def K}, while the theorems are stated in terms of sections of $\calKbar$.
Because of $\calK=(\Lambda \bT^*V) \otimes 
\tilde\calK$ we may consider $P_{00}$ as
acting on sections of $\tilde\calK$ instead. Recall that elements of $\tilde\calK$ are sections of $\Lambda x^a T^*F$ over $F$. There is an obvious bundle map which identifies $\Lambda x^a T^*F$ with  $\Lambda T^*F$. It sends $\calK$ to the space $\calKbar$ of harmonic forms on $F$ for the metric $h_{|F}$. 
Under this identification, $P_{00}$ turns into the operator 
\begin{equation}\label{eqn:D_V}
D_V = \frakd + \frakd^* + xD_x 
+x^a \Pi {\bf R} \Pi
\end{equation}
 since the term $A$ in \eqref{eqn: def A} arose from the commutators $[x\partial_x,x^{a|L|}]$ when introducing the scaling factors $x^a$, and in $\calKbar$ these factors are eliminated again.

\begin{proof}[Proof of Theorems \ref{Thm1} and \ref{Thm2}]
We first prove the Fredholm theorem for $D_M$. Recall that $D_M=x^{-a}P$ where $P$ is $\Pi$-split of order 1, for the bundle $E=\Lambda\cphiT^*M$ and $\calK\subset\Cinfty(F,E)$ as defined after \eqref{eqn:def K}.   Theorem  \ref{thm:fredholm}  shows that $P:x^\alpha H^1_\split(M,E,\dvol_b) \to x^\alpha L^2(M,E,\dvol_b)$ is Fredholm if $\alpha-a\not\in -\imspec (P_{00})$. Over the interior of $M$, we may identify the bundles $\Lambda T^*M$, with metric $g$, with the bundle $E$, with regular metric on all of $M$. Therefore, we get that $D_M:x^\alpha H^1_\split(M,\Lambda T^*M,\dvol_b) \to x^{\alpha-a} L^2(M,\Lambda T^*M,\dvol_b)$ is Fredholm. Setting $\gamma=\alpha-a$ gives the claim.

The Fredholm claim for $\Delta_M$ is proved in the same way, using the condition $[D_B,\Pi]=0$ to get that $\Delta_M=x^{-2a}T$ for a $\Pi$-split operator $T$ of order 2 (see \eqref{eqn:Delta_M} and the remarks following it). These remarks also show that the condition implies $T_{00}=(P_{00})^2+O(x^{2a})$, so the indicial families of $T_{00}$ and $(P_{00})^2$ coincide, implying that
$-\imspec(T_{00}) = -\imspec(P_{00})=-\imspec(D_V)$. Therefore, the condition on $\gamma$ is the same as for $D_M$. 

Finally, the polyhomogeneity claims in Theorems \ref{Thm1} and \ref{Thm2} follow from the second statement in Theorem \ref{thm:regularity}, using that 
$x^w(\log x)^k$ is in $L^2(M,E,\dvol_b)$ iff $\Re w>0$.
\end{proof}



\begin{thebibliography}{1}

\bibitem{GH1} D. Grieser, E. Hunsicker, {\em Pseudodifferential operator calculus for generalized Q-rank 1 locally symmetric spaces, I}, Journal of Functional Analysis 257 (2009) 3748Ð3801.
\bibitem{GH2} D. Grieser, E. Hunsicker, {\em Techniques for the b-calculus}, in preparation.
\bibitem{GH3} D. Grieser, E. Hunsicker, {\em Pseudodifferential operator calculus for generalized Q-rank 1 locally symmetric spaces, II}, in preparation.
\bibitem{Gbbc} D. Grieser,  {\em Basics of the b-calculus}, in
  J.B.Gil et al. (eds.), {\em Approaches to Singular Analysis}, 30-84,
  Operator Theory: Advances and Applications, 125.
  Advances in Partial Differential Equations,
Birkh\"auser, Basel, 2001.

\bibitem{HHM} T. Hausel, E. Hunsicker, and R. Mazzeo, {\em
Hodge cohomology of gravitational instantons}, Duke Mathematical
Journal {\bf 122} (2004), no.3, 485-548.


\bibitem{MaMe}  R. Mazzeo and R. Melrose, {\em Pseudodifferential
operators on manifolds with fibred boundaries} in {\em ``Mikio Sato:
a great Japanese mathematician of the twentieth century.''}, Asian J.
Math. {\bf 2} (1998) no. 4, 833--866.

\bibitem{Me-aps} R. Melrose, {\em The Atiyah-Patodi-Singer index theorem},
A.K. Peters, Newton (1991).

\bibitem{Me-icm} R. Melrose, {\em Pseudodifferential operators, corners and singular limits},  Proc. Int. Congr. Math., Kyoto/Japan 1990, Vol. I, 217-234 (1991).

\bibitem{Me-mwc}  R. Melrose, {\em Differential analysis on manifolds with corners},
in preparation, partially available at http://www-math.mit.edu/$\sim$rbm/book.html.

\bibitem{JM} M\"uller, J\"orn, {\em A Hodge-type theorem for manifolds with fibered cusp metrics}, Geom. Funct. Anal. 21 (2011), no. 2, 443Ð482.

\bibitem{muller} W. M\"uller, {\em  Manifolds with cusps of rank $1$},
Lecture Notes in Math., vol. 1244, Springer-Verlag, New York (1987).

\bibitem{Va} B. Vaillant, {\em Index and spectral theory for manifolds
with generalized fibred cusps}, Ph.D. thesis, Univ. of Bonn, 2001.
arXiv:math-DG/0102072.


\end{thebibliography}
\end{document}